\documentclass[a4paper,10pt]{article}

\usepackage[fleqn]{amsmath}
\usepackage{amssymb}
\usepackage{amsthm}
\usepackage{stmaryrd}
\usepackage{bbm}
\usepackage[colorlinks]{hyperref}
\usepackage{mathtools}
\usepackage[numbers]{natbib}
\mathtoolsset{showonlyrefs}
\usepackage{enumitem}
\usepackage{subcaption}
\usepackage{tikz}
\usepackage{dsfont}
\usepackage[T1]{fontenc}
\usepackage[utf8]{inputenx}
\usepackage{graphicx}
\usepackage{geometry}
\usepackage{leftidx}
\usepackage{color}
\usepackage{pifont}
\usepackage{babel}

\usepackage{thmtools}

\renewcommand{\P}{\mathbb{P}}

\newcommand{\bE}{\ensuremath{\mathbb{E}}}

\newcommand{\bN}{\ensuremath{\mathbb{N}}}

\newcommand{\bP}{\ensuremath{\mathbb{P}}}

\newcommand{\bR}{\ensuremath{\mathbb{R}}}

\newcommand{\bZ}{\ensuremath{\mathbb{Z}}}

\newcommand{\ind}{\ensuremath{\mathbbm{1}}}

\newcommand{\cF}{\ensuremath{\mathcal{F}}}

\usepackage{enumitem, hyperref}
\makeatletter
\def\namedlabel#1#2{\begingroup #2 \def\@currentlabel{#2} \phantomsection\label{#1}\endgroup }
\makeatother

\theoremstyle{plain}
\newtheorem{theorem}{Theorem}[section]  
\newtheorem{proposition}[theorem]{Proposition}

\newtheorem{lemma}[theorem]{Lemma}

\theoremstyle{definition}

\theoremstyle{remark}

\numberwithin{equation}{section}

\title{The truncation property and continuity for \\ the long-range contact process on $\mathbb{Z}^d$}

\author{Stein Andreas Bethuelsen
\thanks{University of Bergen, All\'egaten 41, 5020 BERGEN, Norway\\  Email: stein.bethuelsen@uib.no}
 \and Frank Namugera
\thanks{Mathematics Department, Makerere University, Kampala, Uganda\\
  Email: frank.namugera@mak.ac.ug}}

\begin{document}
\maketitle

\abstract{We consider a general class of contact processes on $\mathbb{Z}^d$ with potentially long-range interactions. 
By adapting  well established renormalization arguments to the long-range setting we extend by now classical results for finite-range processes to this more general setting.  
Particularly, we provide general conditions  on the decay of the interactions under which a supercritical process remains supercritical  after truncation of the interaction parameter at a sufficiently large distance. 
Further, for the family of parameters satisfying this latter truncation property, we conclude that the probability of the process to never recover is continuous.
\\

\noindent\footnotesize{{\textbf{AMS-MSC 2020}: 60K35; 82B43}

}


\section{Introduction and main results}\label{sec:Intro}

\paragraph{Introduction.}  
The ordinary, nearest neighbour, contact process is one of the classical interacting particle systems covered in the defining books \cite{LiggettIPS1985,LiggettSIS1999} by Liggett. On the $d$-dimensional integer lattice, it is the Markov process on the state space $\Omega \coloneqq \{0,1\}^{\bZ^d}$ where the state at a location $x\in \bZ^d$ transitions from $1$ to $0$ at rate $\delta \in (0,\infty)$, and, if in state $1$, transmits its state to a neighbour $y\sim x$ in state $0$ at rate $\lambda \in (0,\infty)$. 

Motivation for studying this dynamic process stems among others from the modelling of infection spread in a population, where one interprets the nodes of the network as individuals of the  population and these are healthy if in state $0$ and infected if in state $1$. Another interpretation is that nodes in state $1$ as occupied by an individual and nodes in state $0$ as vacant, providing a simplified model for the evolution of a population. It also has its origin stemming from high energy physics. Today it forms a fundamental building block, and reference model, for a wealth of agent based models, see e.g.\ \cite{Lanchier2024}.

Although the framework of interacting particle systems was rather early defined for general interactions, 
by the time of Liggett's first book \cite{LiggettIPS1985}, most results for the ordinary contact process were restricted to the one-dimensional lattice. First with the groundbreaking work of Bezuidenhout and Grimmett in \cite{bezuidenhout1990critical} the ordinary contact process on $\bZ^d$, $d\geq 2$, obtained a detailed  and deep mathematical analysis, which not much later was extended to a broader class of (finite range) processes  \cite{bezuidenhout1994critical}. Since then the study of the ordinary contact process has witnessed a flourishing development to cover more general, and realistic, modelling settings, as partly summarized in \cite{LiggettSIS1999} and in the recent books \cite{Lanchier2024,Valesin2024}.

\textbf{The long-range contact process.} 
This paper is devoted to the study of the so-called long-range contact process (LRCP) on $\bZ^d$, $d\geq 1$, where we allow for unbounded interactions.  
That is, as above, the state at a location $x\in \bZ^d$ transitions from $1$ to $0$ at rate $\delta$. However, we are provided with  a collection of non-negative numbers $\lambda=(\lambda_{x,y})_{x,y \in \bZ^d}$ so that, if $x$ is in state $1$, it transmits its state to any $y \in \bZ^d$ in state $0$ at rate $\lambda_{x,y}$. More formally, this is the interacting particle system $(\eta_t)_{t\geq0}$  on $\Omega$ specified by its pre-generator $L_{\lambda} \colon \mathcal{C}(\mathbb{R}) \mapsto \mathcal{C}(\mathbb{R})$ given by 
\begin{equation}\label{eq contact generator LR}
L_{\lambda} f(\omega) := \sum_{\substack{x \in V \\ \omega(x)=1}} \left(   \delta[f(\omega^{x \leftarrow 0})-f(\omega)]+ \sum_{y\neq x} \lambda_{x,y} [f(\omega^{y \leftarrow 1})-f(\omega)] \right),
\end{equation}
for $ \omega \in \{0,1\}^{\bZ^d}$. 
Here, for $z \in V$ and $i \in \{0,1\}$, we denote by $\omega^{z \leftarrow i}$ the configuration where $\omega^{z \leftarrow i}(z)=i$ and $\omega^{z \leftarrow i}(x)=\omega(x)$ for any $x\neq z$. 

For the LRCP to be well-defined we also need to impose some conditions on the parameter $\lambda$, which unless otherwise specified is assumed to satisfy:
\begin{enumerate}
    \item[] \namedlabel{item1:ConditionsAB}{a)} $\lambda_{x,y} = \lambda_{x+z,y+z}$ for all $x,y,z \in \bZ^d$ (translation invariant) 
    \item[] \namedlabel{item2:ConditionsAB}{b)} $ \sum_{y \neq o} {\lambda_{o,y}} < \infty$ (summability)
    \end{enumerate}
    Further, for any $y\in \bZ^d$ we set $\lambda_{y,y}=0$ as this rate has no influence on the dynamics.  
Note that, from these conditions, it follows that  
\begin{equation}\label{eq:lambda_infty}
\lambda_{\infty} \coloneqq \sup_{x \in \bZ^d}    \sum_{y \neq x} \lambda_{x,y} < \infty
\end{equation}
and thus,  by \cite[Thm I.3.9]{LiggettIPS1985}, the corresponding stochastic process is well-defined. 

We denote by $\Lambda$ the space of infection parameters $\lambda= (\lambda_{x,y})_{x,y \in \bZ^d}$ for which $\lambda_{x,y}>0$ for some $x,y \in \bZ^d$ and satisfying Conditions \eqref{item1:ConditionsAB} and \eqref{item2:ConditionsAB}. Clearly the ordinary contact process where $\lambda_{x,y}= \beta \ind_{\{\|x-y\|=1\}}$ for some $\beta\in(0,\infty)$ falls within this setting. Here $\|\cdot \|$ denotes the $\ell_{1}$-norm, and for concreteness and convenience we stick with this norm also in the following. Further, the  extension to having  translation invariant and finite range interactions are also included. The latter means that there is some range $R\in \bN$ so that $\lambda_{x,y}=0$ whenever  $\| y-x \|\geq R$.
A natural example of an infinite range model is to consider $\lambda\in \Lambda$ for which 
\begin{equation}\label{eq:examplesInteraction}
\lambda_{o,y} = \beta  \|y\|^{-\alpha} \quad \text{ for some }\alpha>d \text{ and }\beta \in (0,\infty).
\end{equation}

\textbf{Parameter regimes.} We write  $\bP_{\lambda,\delta}$ for the law of the LRCP with parameters $\lambda \in \Lambda$ and $\delta \in (0,\infty)$,  and let $(\eta_t^x)$ denote this process initiated with only $x\in \bZ^d$ infected, i.e.\ in state $1$. 
Further, we denote the probability of the (infection) process to never die out by 
\begin{equation}
\phi(\lambda,\delta) \coloneqq \bP_{\lambda,\delta}( \eta_t^o \neq \underline{0} \text{ for any } t>0),
\end{equation}
where $\underline{0} \in \Omega$ denotes the all 0's configuration and $o \in \bZ^d$ denotes the origin. By standard arguments that we discuss more carefully in Section \ref{sec:preliminiaries} when introducing the graphical construction of the LRCP,  the function $\phi(\cdot,\cdot)$ is non-increasing in its second argument (and non-decreasing in its first argument). 
Therefore, for any $\lambda \in \Lambda$, we may consider the \emph{critical value}
\begin{equation}
\delta_{c} =\delta_{c}(\lambda)\coloneqq \inf\{\delta>0: \phi(\lambda,\delta)=0\} = \sup\{\delta>0: \phi(\lambda,\delta)=0\}. 
\end{equation}
Using this, as e.g.\ in \cite{Swart2018}, we say that the LRCP with parameters $\lambda$ and $\delta>0$ is \emph{critical} if $\delta=\delta_c(\lambda)$, that it is \emph{subcritical} if $\delta>\delta_c(\lambda)$ and \emph{supercritical} if  $\delta<\delta_c(\lambda)$. 

Note that, since the process with parameters $\lambda\in \Lambda$ and $\delta>0$ is equivalent to the process with parameters  $\delta^{-1}\lambda \in \Lambda$ and $1$ upon a scaling of time, these definitions also agree with the possibly more usual ones in terms of the infection parameter $\lambda$ as e.g.\ in \cite{AizenmanJung2007}. Moreover, it is not difficult to see that $\delta_\mathsf{c}(\lambda) \in (0,\infty)$. Indeed, since $\lambda$ is assumed to be summable, the finiteness of $\delta_\mathsf{c}(\lambda)$ follows from a simple comparison argument with a branching process. Further, since the nearest neighbour contact process on $\bN$ may with positive probability never die out at sufficiently low recovery rates, it immediately follows that $\delta_\mathsf{c}(\lambda)>0$.

There are equivalent descriptions of the above  defined parameter regimes  in terms of the expected cluster size for the process started from only the origin infected, or the so-called upper invariant measure, which we detail next. More precisely,  
with $\lambda \in \Lambda$ and $\delta>0$, the \emph{susceptibility} is defined as
\begin{equation}
 \chi(\lambda,\delta)\coloneqq \int_0^{\infty} \bE_{\lambda,\delta}\left( \left| \{x \in \bZ^d \colon \eta_t^o(x)=1\} \right| \right) dt. 
\end{equation} 
Further, for $i=0$ or $1$, let $\delta_{\underline{i}}$ denote the probability distribution on $\Omega$ that concentrates on the all $i$'s configuration $\underline{i} \in \Omega$. 
Clearly, the LRCP is stationary with respect to the measure $\delta_{\underline{0}}$. 
Again by standard monotonicity arguments that we briefly describe in Section \ref{sec:preliminiaries}, it follows that the LRCP started from $\delta_{\underline{1}}$ converges weakly 
towards the upper invariant measure, denoted here by $\bar{\nu}_{\lambda,\delta}$, and this measure is also a stationary distribution for $(\eta_t)$. 

The following statement summarizes classical relations between $\phi(\lambda,\delta)$,  $\chi(\lambda,\delta)$ and $\bar{\nu}_{\lambda,\delta}$ that in turn give further characterizations of the parameter regimes. 

\begin{proposition}[Proposition 1.1 and Theorem 1.2 in \cite{AizenmanJung2007}, and Theorem 1 in \cite{Swart2018}] \label{prop:SubCritExpDecay}
    Consider the LRCP  on $\bZ^d$, $d\geq 1$, with $\lambda \in \Lambda$ and $\delta>0$. 
\begin{enumerate}
    \item[i)] The process is subcritical if and only if $\chi(\lambda,\delta)<\infty$ and there exists some $\tau<0$ such that
    \begin{equation}
     \inf_{t >0 } \frac{1}{t} \log  \bE_{\lambda,\delta} \left[  \left| \left\{ x\in \bZ^d \colon \eta_t^o(x)=1\right\} \right| \right] = \tau.
    \end{equation}
     Moreover, if the process is subcritical, then $\bar{\nu}_{\lambda,\delta} = \delta_{\underline{0}}$.
    \item[ii)] If the process is critical, then $\chi(\lambda,\delta)=\infty$.
\item[iii)]   If the process is supercritical, then  $\chi(\lambda,\delta)=\infty$ and $\bar{\nu}_{\lambda,\delta} \neq \delta_{\underline{0}}$.
\end{enumerate}
    \end{proposition}
    
  Note that the  statements of   \cite{AizenmanJung2007} and \cite{Swart2018} referred to above were proven in a more general setting of the LRCP on transitive graphs.

\paragraph{Main results.} Previously, slight variants of the above defined LRCP have been studied by several,  e.g.\ in  \cite{Bramson1981,Can2015,Gomes2022,Jahnel2025,SeilerSturm2023}. 
The scope of these papers, however, is mainly to establish qualitative results, i.e.\ providing parameter regimes for which $\delta_c(\lambda) \in (0,\infty)$ or $\delta_c(\lambda) \notin (0,\infty)$, respectively. 
One motivation for our work  has been to develop a machinery that can be used to provide more quantitative results through the use renormalization techniques for such (and other) long-range contact processes - see also the discussion in Section \ref{sec:applications}.

The main results of this paper give a further characterization of the parameter regimes defined above. 
For this, following the convention of  \cite{baumler2023continuity}, we next introduce the  concept of \emph{resilience} with respect to truncation. For this, given $\lambda \in \Lambda$  and $k \in \bN$, we let \begin{equation}\label{eq:lambdaKcK}
     \leftidx{_{k}}\lambda_{x,y} \coloneqq \lambda_{x,y} \ind_{\{\|x-y\| \leq k\}}, \quad x,y \in \bZ^d.
\end{equation}
 The infection parameter $(\leftidx{_{k}}\lambda)$ is thus \emph{the truncation of $\lambda$ at level $k$}. 
An infection parameter $\lambda \in \Lambda$  is said to be resilient if $\lim_{k \rightarrow \infty}     \phi(\leftidx{_{k}}\lambda,\delta)>0$ whenever 
    $\phi(\lambda,\delta)>0$. 
 
Our first result states that, for a resilient infection parameter, the LRCP dies out at criticality. This generalize the results of \cite{bezuidenhout1990critical,bezuidenhout1994critical} to the long-range setting.

\begin{theorem}[Extinction at criticality]\label{thm:resilient2}
If $\lambda \in \Lambda$ is resilient, then the  LRCP with parameters $\lambda$ and $\delta>0$ is critical if and only if $\chi(\lambda,\delta)=\infty$ and $\phi(\lambda,\delta)=0$. \end{theorem}

Note that, as a consequence of Theorem \ref{thm:resilient2} and Proposition \ref{prop:SubCritExpDecay}, for any $\lambda \in \Lambda$ resilient,  the LRCP with parameters  $\lambda$ and $\delta$ is supercritical if and only if $\phi(\lambda,\delta)>0$. It also follows that, if the process is critical, then $\bar{\nu}_{\lambda,\delta} = \delta_{\underline{0}}$.

Our next theorem is inspired by  \cite[Theorem 1.4 and Corollary 1.5]{baumler2023continuity}, who recently proved a similar statement for long-range bond percolation on $\bZ^d$. 

\begin{theorem}[Continuity]\label{thm:resilient} 
Let $\lambda\in \Lambda$ be resilient. Then, for any sequence  $\delta^{(n)}\rightarrow \delta  \in [0,\infty)$ and any sequence $(\lambda^{(n)})$ of elements in $\Lambda$ satisfying 
\begin{equation}\label{eq:condTheoremResilient} 
\lim_{n\rightarrow \infty} \sum_{y \in \bZ^d} \sup_{m\geq n}| \lambda^{(m)}_{o,y} - \lambda_{o,y}| =0,
\end{equation}
it holds that
    $\lim_{n \rightarrow \infty} \phi(\lambda^{(n)},\delta^{(n)})=\phi(\lambda,\delta)$. 
 In particular,  $\delta_c(\lambda^{(n)}) \rightarrow \delta_c(\lambda)$ as $n \rightarrow \infty$. 
\end{theorem}

Clearly, any $\lambda \in \Lambda$ of finite range is resilient. Our last result of this section shows that rather general families of long-range infection parameters are resilient. 
For this, we say that $\lambda \in \Lambda$ is \emph{symmetric} if for all $y =(y_1,\dots,y_d)\in \bZ^d$ it holds that 
     $\lambda_{o,y} = \lambda_{o,\tilde{y}}$    
       for any $\tilde{y} \in \bZ^d$ obtained by either changing the signs of any of the $d$ components of $y$ or after permuting them.  

\begin{theorem}[Truncation]\label{thm:truncation} $ $
\begin{enumerate}
\item[a)] The symmetric case: 
All elements of $\lambda \in \Lambda$ that are symmetric and satisfy $\lambda_{o,y} \simeq \|y\|^{-\alpha}$ for some $\alpha>d$  are resilient. 
\item[b)] The non-symmetric case:
All elements of $\lambda \in \Lambda$ that satisfy $\lambda_{o,y} = \mathcal{O}(\|y\|^{-\alpha})$ for some $\alpha>2d+1$  are resilient.\\ 
\end{enumerate}
\end{theorem}
Here,  as standard, we write $\lambda_{o,y}=  \mathcal{O}(\|y\|^{-\alpha})$  if there are  $\beta<\infty$ and $N\in \bN$ so that $\lambda_{o,y} \leq \beta \|y\|^{-\alpha} \text{  for all  } y \in \bZ^d$ such that $\|y \| \geq N$.  Moreover, $\lambda_{o,y}\simeq \|y\|^{-\alpha}$ for some $\alpha>d$ if $\text{for some  }\beta>0$ and $R \in \bN$, it holds that 
$ \beta^{-1} \leq \frac{\lambda_{o,y}}{\|y\|^{-\alpha}} \leq \beta  \text{ whenever }  \|y\| \geq R.$

\paragraph{Outline of the paper.} 
In the next section we discuss briefly our main results and how they relate to the existing literature. 
In Section \ref{sec:preliminiaries} we detail the graphical representation of the LRCP and some of its basic consequences. Then, in Sections  \ref{sec:block}-\ref{sec:proofs}, we present the proofs of  the main results. Particularly, in Section \ref{sec:block}, which is devoted to the symmetric case, we provide first the details to the finite space-time condition that constitutes the key step of the aforementioned renormalization argument and from which we then derive Theorem \ref{thm:truncation}a). Here we strongly rely on ideas sketched in \cite[Section 7]{grimmett2002directed}  for the study of a related long-range oriented percolation model. The proof for the non-symmetric case of Theorem \ref{thm:truncation}b) furthermore builds on ideas from \cite{bezuidenhout1994critical} and is presented in Section \ref{sec:non-sym}. In Section \ref{sec:proofs} we present the proofs of Theorem \ref{thm:resilient2} and Theorem \ref{thm:resilient}. They rely on the fact that  any supercritical contact process having finite-range interactions satisfy a finite space-time condition akin to that referred to above, as follows by the results of \cite{bezuidenhout1994critical}. In the last subsection we also discuss briefly some further applications of the resilient property, i.e.\ local survival, complete convergence and the shape theorem.

\section{Discussion}\label{sec:applications} 

\textbf{The tail-regularity conditions.}  
The assumptions on $\lambda \in \Lambda$  that we pose in Theorem \ref{thm:truncation} can be slightly weakened. Indeed, as seen in the proof of Proposition \ref{prop:tailbound} below, the fundamental property that we require in the symmetric case and the proof of Theorem \ref{thm:truncation}a)  is that the bound \eqref{eq:tail_b} holds. This is reminiscent of Condition III  as given in \cite[Section 7]{grimmett2002directed} for the study of a long-range oriented percolation model.

Also note that, whenever  $\alpha>2d+1$, the tail-regularity condition of  Theorem \ref{thm:truncation}a) is stronger than that  assumed in   Theorem \ref{thm:truncation}b) for the non-symmetric case.

The fundamental property that we require in the non-symmetric case is that the corresponding LRCP has at most linear growth of the  LRCP in the sense of Proposition \ref{prop:fstcWS1}. 
The assumptions for this to hold can also be weakened somewhat, see e.g.\ the discussion directly after the statement of Theorem 1.7 in  \cite{chatterjee2016multiple} for more on this for similar bounds obtained for the corresponding long-range first-passage percolation model (i.e.\ the LRCP with $\delta=0$.)

For comparison, a statement akin to Theorem \ref{thm:truncation} was in \cite{MeesterSteif1996}  shown to hold for a long-range bond percolation  on $\bZ^d$, $d\geq 2$, with a parameter $p=(p_{x,y})$ decaying at exponential rate. For this they also assumed that $p_{x,y}$ depends only on $\|x-y\|$ and thereby enforcing that it is translation invariant, symmetric and irreducible. Moreover, by a different method, \cite[Theorem 1.8]{Berger2002}  concluded  the same resilient property for $(p_{x,y})$ satisfying $p_{x,y}\simeq \|y-x\|^{-\alpha}$ with $\alpha <(d,2d)$.  Recently, the results of \cite{MeesterSteif1996} was improved by \cite{baumler2023continuity} to hold under the requirement that $p_{x,y} = \mathcal{O}(\|x-y\|^{-2d})$, still assuming translation invariance, symmetry and irreduciblity. 

\textbf{The truncation property.} 
The question whether an infection parameter is resilient, also known as the truncation property, has earlier been considered for the LRCP with  non-summable infection parameter  \cite{AlvesETAL2017,CamposLima2022,EnterLimaValesin2016}, i.e.\ $\lambda_{\infty}=\infty$. Although the model itself is then ill-defined, the infection parameter may still be resilient in the sense that $\lim_{k \rightarrow \infty} \phi(\leftidx{_{k}}\lambda,\delta)>0$ for every $\delta>0$. For instance, \cite[Theorem 2]{EnterLimaValesin2016} states that if $\sum_{i\geq 1} \lambda_i=\infty$ then the LRCP on $\bZ^d$, $d \geq 2$, satisfies this truncation property whenever the interaction function is given by 
\begin{equation}
\lambda_{x,y} = \left\{\begin{array}{cc} \lambda_{\|i\|} & \text{ if }y=x+ i e_m \text{ with } i \in \bZ; \\0 & \text {otherwise,}\end{array}\right.
\end{equation}
 where $e_1,\dots,e_d$ are the unit vectors in $\bZ^d$. In particular, by this and  Theorem \ref{thm:truncation}, the truncation property holds for infection parameters $\lambda$ of the form  \eqref{eq:examplesInteraction} for any $\alpha>0$. 
 
The papers \cite{AlvesETAL2017,CamposLima2022} provide extensions of the results in  \cite{EnterLimaValesin2016}, particularly to cover cases when the underlying graph is $\bZ$. For instance,  \cite[Theorem 3]{AlvesETAL2017} show that the truncation property holds whenever  there is an $\epsilon>0$ such that $\lambda_{o,y}>\epsilon$ for infinitely many $y$'s. The paper \cite{CamposLima2022} prove the truncation property in a regime corresponding to $\sum_{i\geq 1} \lambda_i^2=\infty$, albeit for an oriented percolation model.

 \textbf{Locality results.} 
Our main results concern the LRCP on $\bZ^d$, $d\geq 1$, and provide general conditions under which it is, in a certain sense, non-sensitive with respect to local changes of the parameter within the supercritical regime. In principle, one could also consider how sensitive the process is with respect to changes in the graph. Results in this direction have previously been obtained in \cite{bezuidenhout1990critical,bezuidenhout1994critical} for finite range models. There it is shown that, if $\phi(\lambda,\delta)>0$ for this model on $\bZ^{d+1}$, $d\geq 1$, then as concluded in \cite[Theorem 2.8]{bezuidenhout1994critical}, the same holds true when restricted to a sufficiently thick one-dimensional slab $[-k,k]^d \times \bZ$. In fact, by monotonicity  this result immediately extends to the LRCP when $\lambda \in \Lambda$ is resilient since it then, by definition, satisfies the truncation property.  
 
For other graphs than $\bZ^d$ the only other continuity results we are aware of are for the ordinary contact process on regular trees \cite{Morrow1994}, but see also the discussion in \cite{Valesin2024} for related results in the setting of the contact process on finite (and random) graphs. 
Such questions have, however, received more attention for ordinary percolation, see in particular \cite{ContrerasMartineauTassion2023,DiskinEtAl2026,EasoHutchcroft2023}. It would be very interesting to explore similar extensions for the LRCP.


\textbf{Extension to other attractive interacting particle systems.} 
In this paper we focus solely on the LRCP, but our arguments extends to a much wider class of processes on $\bZ^d$ after relatively minor modifications. Particularly, the work of \cite{bezuidenhout1994critical} included a  general class of continuous-time and discrete-time  attractive spin-flip dynamics with finite-range dependence.   In their general setting, the finite range assumption appears to be crucial. However, the extension presented in this paper can presumably be adapted without much effort to the sub-class of additive spin-flip dynamics for which there is a graphical representation akin to that in Section \ref{sec:preliminiaries} (see e.g.\ \cite[Chapter 3.6]{LiggettIPS1985}) and of which the LRCP is a key example.

Renormalisation arguments reminiscent to those of  \cite{bezuidenhout1990critical,bezuidenhout1994critical} have also been successfully applied to various variants of the contact process, e.g.\ with aging and other random growth models  \cite{deshayes2014contact,deshayes2025contact}, in a dynamic evolving environment \cite{SteifWarfheimer2008,SeilerSturm2022} or a two-level contact process \cite{Ruibo2022}. Using the methods of the current paper, we expect that these results too may be extended without much further ado to the possibly more natural setting of unbounded interactions.



\section{The graphical representation and its consequences}\label{sec:preliminiaries}

We start by presenting the so-called graphical construction of the LRCP, first introduced in \cite{harris1978additive} for the ordinary contact process and which we here adapt to the long-range setting. 

Fix  $\lambda \in \Lambda$ and $\delta>0$. Independently, for each $x \in \bZ^d$, let $(H_x)$ be a Poisson process with rate $\delta$ on $\bR$, and for each ordered pair $(x,y) \in \bZ^d \times \bZ^d$, let $(I_{x,y})$ be a Poisson process with rate $\lambda_{o, y-x}$ on $\bR$. 
 An event of $(H_x)$ represents a potential \emph{healing event} where the state at $ x $ at that time is set to $ 0 $,  whereas an event of $(I_{x,y})$ represents a potential \emph{infection event}, where the state at $ y $ will be set to $ 1 $ if the state at either $  x $ or $ y $  is $ 1 $.

For each $x,y \in \bZ^d$ and $0 \leq s\leq t$, we say that $(x,s)$ is connected to $(y,t)$ by an \emph{active path}, written $(x,s) \rightarrow (y,t)$, if and only if there exists a càdlàg path in $\bZ^d \times \bR$, starting at $(x,s)$ and ending at $(y,t)$, that goes forwards in time without hitting any healing event  and that may cross to another vertex at the instance of an infection event  in the prescribed direction of the ordered pair. That is, there exists a sequence $x=x_0,x_1,\dots,x_n=y$ in $\bZ^d$ and times $s=t_0<t_1<\dots<t_{n+1} = t$ such that, for each $i=0,\dots,n$, there are no healing events at $x_i$ within time $[t_i,t_{i+1}]$, but there is an infection event at $(x_i,x_{i+1})$ within the same time window. 

Based on the processes introduced above, we now describe how to construct the LRCP. Particularly, for $A \subset \bZ^d$, consider the process $(A_t^A)_{t\geq 0}$ evolving on the subsets of $\bZ^d$ given by
\begin{equation}\label{def:GP_CP}
x \in A_t^A \:  \iff \: \exists \: y \in A \colon \:(y,0) \rightarrow (x,t)
\end{equation} 
Then, the process $(\tilde{\eta}_t)_{t\geq0}$ evolving on $\Omega$ and defined by $\tilde{\eta}_t(x) = 1$ if and only if $(x,t) \in A_t^A$ equals,  in distribution, the LRCP defined via the generator in \eqref{eq contact generator LR} and with initial distribution $\eta_0(x)=1$ if and only if $x \in A$. For a mathematical justification of this we refer to \cite[Chapter 3]{LiggettIPS1985}.
Due to this equivalence, in the following we also refer to the process  defined by  \eqref{def:GP_CP} as the LRCP. Moreover, we write $\bP_{\lambda,\delta}$ to denote its distribution on the (sufficiently large) probability space $(\Omega,\cF,\bP_{\lambda,\delta})$ containing all the processes of the graphical construction introduced above. We also write $(A_t^x)$ whenever $A=\{x\}$ with $x\in \bZ^d$. 

For any $A \subset \bZ^d$ and (reasonable) space-time set $D \subset \bZ^d\times[0,\infty)$, we write $( \leftidx{_{D}}A^{A}_{s})_{s\geq0}$ for the LRCP restricted to $D$, that is,
\begin{equation}
x \in \leftidx{_{D}}A_t^A \:  \iff \: \exists \: y \in A \colon \:(y,0) \rightarrow_D (x,t),
\end{equation}
where $(y,0) \rightarrow_D (x,t)$ denotes the event that there exists an active path from $(y,0)$ to $(x,t)$ fully contained in $D$.
 Later we will often consider cases where $D= C \times [0,T)$ for some $C \subset \bZ^d$ and $T \in (0,\infty)$ and then write  $(\leftidx{_{C}}A^{A}_s)_{s \in [0,T]}$ instead of $(\leftidx{_{C\times[0,T]}}A^{A}_s)_{s\geq 0}$. In fact, we will follow the convention  of \cite[Chapter 1]{LiggettSIS1999} and simply write  $(\leftidx{_{L}}A^{A}_s)_{s\geq 0}$ when  $C=B_L(o)$. We further denote by $\cF_{D}$ the  smallest $\sigma$-algebra that contains all events associated with the graphical construction within $D$.


We now turn to the many useful properties that may be derived for the LRCP using the above graphical representation. 
Although they are quite standard, for completeness and later reference, we next summarize those that will be used frequently in the proofs of our main results presented in the later sections. These properties are derived exactly as for the nearest neighbour model, and we refer to  \cite[Chapter 1]{LiggettSIS1999}, and references therein, for their formal proofs. 

\begin{enumerate}

  \item \textbf{Scale invariance:} 
A first consequence to note  is that, for any $c>0$, the LRCPs with parameters $\lambda \in  \Lambda$ and $\delta >0$ is equivalent to the LRCP with parameters $c\lambda$ and $c\delta>0$ after a reparametrization of time. To see this, note that the first process may be constructed using the  Poisson processes $(H_x)$ and $(I_{x,y})$, whereas for the second we may use $(cH_x)$ and $(cI_{x,y})$, giving a coupling with the claimed property.

 \item  \textbf{Monotonicity:}   
 For $A_1 \subset A_2$, both subset of $\bZ^d$, the graphical representation immediately provides a monotone coupling of $(A_t^{A_1})$ and $(A_t^{A_2})$. Indeed, by using the same healing and infection events, it follows directly from the definition that 
    \begin{equation}
            \bP_{\lambda,\delta} \left( A_t^{A_1} \subset  A_t^{A_2}\: \forall \: t \in [0,\infty) \right) = 1.
    \end{equation}   
    We write $\lambda^{(1)} \leq \lambda^{(2)}$ for two elements in $\Lambda$ when $\lambda_{o,y}^{(1)} \leq \lambda_{o,y}^{(2)}$ for all $y \in \bZ^d$.  
Then, in a similar vein as above, for any     $\lambda^{(1)} \leq \lambda^{(2)}$ 
     and $\delta_2\leq \delta_1$, we can construct a monotone coupling $\widehat{\bP}$ of $\bP_{\lambda_1,\delta_1}$ and $\bP_{\lambda_2,\delta_2}$ such that, for any $A \subset \bZ^d$,
    \begin{equation}
        \widehat{\bP}\left( A_t^{(A,1)}(x) \leq A_t^{(A,2)} \: \forall \: t \in [0,\infty) \right) = 1,
    \end{equation}
    where $A_t^{(A,i)}$ denotes the LRCP with initial state $A$ and infection parameter $\lambda_i$, $i=1,2$, respectively.
     Indeed, consider for each $x\in \bZ^d$ independent Poisson precesses  $(H_x^{(2)})$ and $(\hat{H}_x)$ of rate $\delta_2$ and $\delta_1-\delta_2$, where the latter represents additional healing events for the  $A_t^{(A,1)}$ process.  Further,  for each $x,y \in \bZ^d$, let $(I_{x,y}^{(1)})$ be a Poisson process with rate $\lambda_{o, y-x}^{(1)}$ and include an additional Poisson process $(\hat{I}_{x,y})$ with rate $\lambda_{x,y}^{(2)} - \lambda_{x,y}^{(1)} \geq 0$, representing additional infection events for the  $A_t^{(A,2)}$ process.  
 From this one can construct $A_t^{(A,1)}$ as above  using the infection events of $(I_{x,y}^{(1)})$ and the healing events of  $(H_x^{(2)})$ and $(\hat{H}_x)$, whereas $A_t^{(A,2)}$ is constructed using the infection events of $(I_{x,y}^{(1)})$ and $(\hat{I}_{x,y})$, and the  healing events of  $(H_x^{(2)})$.  
 This will later be referred to as the natural monotone coupling of the LRCP.

\item \textbf{The upper invariant measure:} 
Using the graphical representation, one can give a more explicit construction of the upper invariant measure $\bar{\nu}_{\lambda,\delta}$.
 Indeed, for $A \subset \bZ^d$ finite, we have that 
\begin{align}
\bP_{\lambda,\delta}(\eta_t^{\bar{1}}(x) = 1 \text{ for some } x \in A) &= \bP_{\lambda,\delta} \left( A \cap A_t^{\bZ^d} \neq \emptyset \right)
\\ &= \bP_{\lambda,\delta} \left( \exists \: y \in \bZ^d \colon (y,-t) \rightarrow (x,0) \text{ for some } x \in A \right),
\end{align}
where the latter equality follows by translation invariance.  Hence, since the upper invariant measure is the distribution towards which $\bP_{\lambda,\delta}(\eta_t^{\bar{1}} \in \cdot)$ converges as $t \rightarrow \infty$,  it follows that 
\begin{align}
&\nu_{\lambda,\delta}(\{\omega \colon \omega(x) = 1 \text{ for some } x\in A)
\\ = &\bP_{\lambda,\delta} \left( \text{for each } t>0 \: \exists \: y \in \bZ^d \colon (y,-t) \rightarrow (x,0) \text{ for some } x \in A \right).
\end{align}
More generally, from this one obtains that, for any $\sigma \in \Omega$ and $A \subset \bZ^d$ finite, 
\begin{align}
&\nu_{\lambda,\delta}(\{\omega \colon \omega(x) = \sigma(x) \text{ for  } x\in A)
\\ = &\bP_{\lambda,\delta} \left( \text{For } x \in A \: \exists \: y \in \bZ^d \colon (y,-t) \rightarrow (x,0) \text{ for each } t<0 \text{ if, and only if }  \sigma(x)=1 \right).
\end{align}

    \item \textbf{Positive association:} Recall the general probability space $(\Omega,\cF,\bP_{\lambda,\delta})$ on which the graphical representation is defined. We associate a partial ordering to $\Omega$ by $\omega \leq \omega'$ if for all $x \in \bZ^d$ it holds that $H_x(\omega') \subseteq H_x(\omega)$, and for all pairs $x,y \in \bZ^d$ it holds that $I_{x,y}(\omega) \subseteq I_{x,y}(\omega')$. An event $E \in \cF$ is said to be a \emph{positive event} if $\omega \in E$ implies that $\omega' \in E$ for all $\omega'$ such that $\omega\leq \omega'$. From the graphical construction it follows that positive event are positively correlated, a results than goes back to \cite{harris1978additive}. That is, for any positive events $E_1$ and $E_2$ in $\cF$, it holds that
    \begin{equation}
    \bP_{\lambda,\delta}(E_1\cap E_2) \geq \bP_{\lambda,\delta}(E_1) \bP_{\lambda,\delta}(E_2).
    \end{equation}

    \item \textbf{Strong Markov property}: The graphical construction  implies that the LRCP satisfies the strong Markov property. That is, if $\tau$ is a finite stopping time with respect to the $\sigma$-fields $(\cF_t)$ where $\cF_t \coloneqq \cF_{\bZ^d \times [0,t]}$, then restricting to the healing and infection events for times $t\geq \tau$, we can readily construct the process $A_t^{A,\tau}$,  for $t\geq \tau$ and with initial state $A$ (at time $\tau$) so that, for any stopping times $\sigma\leq \tau$ and time $t\geq \tau$, we have that 
  $  A_t^{A,\sigma} = A_t^{A_{\tau}^{A,\sigma},\tau}$. 
\end{enumerate}


\section{The symmetric case}\label{sec:block}

In this section we provide the key steps in order to modify the renormalization arguments of \cite{bezuidenhout1990critical} for the ordinary contact process to the long-range setting, restricting to $\lambda \in \Lambda$ symmetric and the statement of Theorem \ref{thm:truncation}a). 
For this, we mostly follow the exposition of \cite{LiggettSIS1999} and argue that the equivalent of  \cite[Theorem 2.12]{LiggettSIS1999}, which corresponds to \cite[Lemma 4.1]{bezuidenhout1990critical}, holds for our more general model with possibly long-range interactions. 
The modifications required for this rely on ideas originating in \cite[Section 7]{grimmett2002directed} for long-range oriented percolation. 
We expand on this and provide the details in the continuous-time setting. 
Further, to cover the case where $\lambda\in \Lambda$ is not necessarily symmetric as in Theorem \ref{thm:truncation}b) we furthermore build on the approach of \cite{bezuidenhout1994critical}. The proof of this latter case is presented in the proceeding section.

Before presenting the first result of this section, we introduce some more notation. 
Firstly, for $L \in [0,\infty)$, $T \in [0,\infty)$ and $R \in [0,\infty)$, we let 
\begin{equation} B_L \coloneqq \bZ^d \cap [-L,L]^d, \quad B_{L,T} \coloneqq B_L \times [0,T], \quad S_{L,T,R} \coloneqq B_{L+R,T} \setminus B_{L,T}.\end{equation} 
Writing $+$ in the superscript of the above defined sets denotes their restriction to the positive coordinates, that is,
\begin{equation}
B_L^+ \coloneqq B_L \cap  [0,L]^d, \quad B_{(L,T)}^+ \coloneqq B_L^+ \times   [0,T].
\end{equation}
Writing $x =(x_1,\dots,x_d)$ for $x \in \bZ^d$, for $i=1,\dots,d$, we also consider the sets
\begin{align}
    &S_{L,T,R}^{i,+} \coloneqq \{ (x,t) \in B_{L+R,T}^+ \setminus B_{L,T}^+ \colon x_i > L \},
    &S_{L,T,R}^{i,-} \coloneqq \{ (x,t) \in B_{L+R,T}^+ \setminus B_{L,T}^+ \colon x_i < -L \}. 
\end{align}
Lastly, for $x \in \bZ^d$ and $A \subset \bZ^d$, we write $A(x)$ for the set $A$ shifted by $x$, that is, 
$A(x) \coloneqq \{ y \in \bZ^d \colon y-x \in A\}.$ 

\begin{theorem}[Finite space-time condition for symmetric interactions]\label{thm:fstc}
Let $\lambda \in \Lambda$ be symmetric and satisfy  $\lambda_{o,y} \simeq \|y\|^{-\alpha}$ for some $\alpha>d$. 
Then, for every $\delta>0$ such that $\phi(\lambda,\delta)>0$, there is $r>1$ so that the following  holds:  for every  $\epsilon > 0$ there are choices of $A \subset \bZ^d$ finite and $L, T \in (0,\infty)$ such that 
\begin{align}
\label{C1}\tag{C1}
 &\P_{\lambda,\delta} \left(\exists x\in B_L^+ \colon  A(x)  \subset \leftidx{_{2L}}{A}{_{T}^{A}} \right)> 1-\epsilon;
 \\ &    \label{C2}\tag{C2}
   \P_{\lambda,\delta}\left(  \exists (x,t) \in S_{L,T,(1+2r)L}^{i,\pm}  \colon A(x)  \subset \leftidx{_{2(1+r)L}}{A}{_{t}^{A}} \right)> 1-\epsilon.
 \end{align}
\end{theorem}

Theorem \ref{thm:fstc}  forms the key ingredient for the aforementioned renormalization argument of \cite{bezuidenhout1990critical}. 
Before presenting its proof, which is given in the proceeding subsection, we show how it implies Theorem \ref{thm:truncation}a).    
 For this, we first present the following application of Theorem \ref{thm:fstc} that constitutes the fundamental building block of the renormalization argument. 

\begin{proposition}\label{prop:renormalization1}
Let $\lambda \in \Lambda$  be symmetric and satisfy  $\lambda_{o,y} \simeq \|y\|^{-\alpha}$ for some $\alpha>d$, and let $\delta >0$ so that $\phi(\lambda,\delta)>0$. Then, for any  $\epsilon > 0$, there exists $A \subset \bZ^d$ and $L, T \in [0,\infty)$   such that
\begin{align}
\label{B1}
 \P_{\lambda,\delta} \left(\exists (y,t) \in D \colon A(y)  \subset \leftidx{_{5L}}{A}{_{6T-s}^{A(x)}} \right)> (1-\epsilon) \quad \text{ for any } (x,s)\in B_{L,T},
 \end{align}
 where $D \coloneqq  [L,3L] \times [-L,L]^{d-1} \times [5T,6T]$.
\end{proposition}

Note that the tail regularity is used in order to guarantee that, in the above inequality and the definition of $D$,  the constants $3, 5$ and $6$ do not depending on $A$, $L$, $T$ nor on $\epsilon$.

\begin{proof}[Proof of Theorem \ref{thm:truncation}a)]
Provided with Proposition \ref{prop:renormalization1}, we can apply the by now relatively standard  renormalization argument as in \cite{bezuidenhout1990critical}, giving a coupling with a supercritical oriented percolation process by suitable choice of $\epsilon$.  See also \cite{LiggettSIS1999}, or the more detailed argument provided in \cite{deshayes2014contact} (albeit for a slightly more general model named \emph{the contact process with aging}).  Indeed, \cite[Lemma 4.10]{deshayes2014contact}  corresponds to Proposition \ref{prop:renormalization1}. From this, the renormalization argument as given in  \cite[Theorem 4.11 and Corollary 4.12]{deshayes2014contact} transfers directly to our setting only after replacing the $[-n,n]^d$-box therein by the set $A$ in Proposition \ref{prop:renormalization1}, and from this we conclude the proof.
\end{proof}

\begin{proof}[Proof of Proposition \ref{prop:renormalization1}]
The argument follows along the same line of reasoning as that in \cite[Lemma 18 and 19]{bezuidenhout1990critical} or \cite[Proposition 2.20 and 2.22]{LiggettSIS1999}. That is, we will apply the finite space-time condition of Theorem \ref{thm:fstc} repeatedly, combined with the strong Markov property and positive association of the LRCP, a finite number of times to direct the infection to a designated area. 

 Let $r>1$ be  as is in Theorem \ref{thm:fstc}. 
Then, for fixed  $\epsilon > 0$, let $\tilde{\epsilon}>0$ be such that $(1-\tilde{\epsilon})^{R_1+R_2} \geq 1-\epsilon$, where we set 
\begin{equation}
    R_1 = 2 (r +2 ) +1 \text{ and } R_2 = 6 (R_1+1).
\end{equation}
Further, fix $A \subset \bZ^d$ finite and $L, T \in (0,\infty)$ such that the inequalities \eqref{C1} and \eqref{C2} of Theorem \ref{thm:fstc} holds with probability $1-\tilde{\epsilon}$. 
Then, consider a space-time point $(x,s) \in B_{a,b}$, where we set $a=R_1 L$ and $b=(R_1+R_2+1)T$.

In the following we will construct a sequence $(x_i,t_i)$ of space-time points iteratively. These will be random, but due to the strong Markov property of the LRCP, the distribution of $(x_{l+1},t_{l+1})$ will only depend on $(x_{l},t_{l})$ and not on previously explored space-time points. 

Firstly, set $x_0 =x$ and $t_0=s$. Then, let $(x_1,t_1)$ be the first (with respect to the temporal order) space-time point in $(x_0,t_0) + S_{L,T,(1+2r)L}$ to fulfil the event within the probability in \eqref{C2} and so that 
\begin{equation}x_1 \cdot e_1 \geq x_0 \cdot e_1 + L \text{ and } x_1\cdot e_i \in [-a,a] \text{ for each } i=2,\dots,d,
\end{equation}
i.e.\ the first coordinate of $x_1$ is at least at distance $L$ from that of $x_0$ to the right, and all other coordinates are within $[-a,a]$. 
Note that, as a consequence of Theorem \ref{thm:fstc}, symmetry of the model and since $(r+1) < R_1$, the existence of  such a point  has probability of at least  $1-\tilde{\epsilon}$. Moreover, since the LRCP is Markovian and $(x_1,t_1)$ is the first such point, the event is contained in $\cF_{B_{2(1+r)L,t_1}}$. 

Using the strong Markov property of the LRCP, we can now iterate this procedure of applying \eqref{C2} successively to find space-time points $(x_l,t_l)_{l=2,\dots,\tau_1}$ located within $(x_{l-1},t_{l-1}) + S_{L,T,(1+2r)L}$. Here, 
\begin{equation}\tau_1 \coloneqq \inf \{ l \geq 1 \colon x_l \cdot e_1 \geq a \}
\end{equation} denotes the first time that we find a space-time point with its first coordinate to the right of $a$. Moreover, we additionally require that, for each $l \leq \tau_1$, 
\begin{equation}\label{eq:C2_1}
    x_l \cdot e_1 \geq x_{l-1} \cdot e_1 + L \text{ and } x_l \cdot e_i \in [-a,a] \text{ for each }i =2,\dots, d. 
\end{equation}
Note that, by the bound in \eqref{C2}, and using the symmetry and translation invariance of the model, the probability of succeeding to find such a space-time point fulfilling the event within the probability in \eqref{C2}  at each iteration step is at least $1-\tilde{\epsilon}$. 

Now, since for each iteration we have that  $x_l \cdot e_1 - x_{l-1} \cdot e_1 \in [L,(1+2r)L]$, we necessarily have that $1<\tau_1 \leq R_1$.  In case $\tau_1<R_1$, we apply \eqref{C1} in a similar vein to find additional space-time points $(x_l,t_l)_{l=\tau_1+1,\dots,R_1}$ located so that  $x_l \in B_L(x_{l-1})$ and $t_l=t_{l-1}+T$. Again, taking advantage of that the bound of \eqref{C1} holds for all symmetric translations of $B_L^+$, we can require that these points remain within $[a,3a]\times[-a,a]^{d-1}$ and still have that we succeed in finding such a point with a probability of at least $1-\tilde{\epsilon}$.  Thus, after exactly $R_1$ iterations we have found a space time point $(x_{R_1},t_{R_1})$ with $x_{R_1} \in [a,3a]\times[-a,a]^{d-1}$ and $t_{R_1} \in [0, (R_1+1) T]$ and so that $A(x_{R+1})  \subset A_{t_{R_1}}^{A}$. Moreover, due to positive association and the strong Markov property, the  probability of finding such a space-time point is at least equal to $(1-\tilde{\epsilon})^{R_1}$.

Thus, so far, we have managed to control the spatial location of an infected copy of the set $A$, as in the statement of the proposition. We next show how the temporal location can be controlled. For this we apply \eqref{C1}, again iteratively, $R_2$ number of times to find space time-points $(x_l,t_l)_{l=R_1,\dots,R_1+R_2}$, again fulfilling that  $A(x_{l})  \subset A_{t_{l}}^{A}$. As above we require that for each iteration the  spatial coordinate is always contained in $[a,3a]\times[-a,a]^{d-1}$. Moreover, since  $t_{l+1}-t_l=T$ for each $l=R_1,\dots R_1+R_2-1$, we necessarily have that $t_{R_1+R_2} \in [t_{R_1} + R_2T, t_{R_1} + (R_2+1)T] \subset [5b,  6b]$. Due to positive association and the strong Markov property, the  probability of finding such a sequence of space-time points is at least equal to $(1-\tilde{\epsilon})^{R_1+R_2}$. Moreover,  for the entire construction, we were always restricted to the graphical representation within $B_{5a,6b}$, and from this we conclude the proof with $L=a$ and $T=b$.
\end{proof}

\subsection{Proof of Theorem \ref{thm:fstc}}

The proof of Theorem \ref{thm:fstc} is derived from several propositions that we first present. These follow partly the exposition as given in \cite{LiggettSIS1999} for the nearest neighbour model. The proofs of the propositions are postponed to later  subsections.

The first proposition  corresponds to  \cite[Propositions I.2.1 and I.2.2]{LiggettSIS1999}, and parts of the introductory argument in the proof of \cite[Theorem I.2.12]{LiggettSIS1999}.

\begin{proposition}\label{prop1.1}
    Let $\lambda \in \Lambda$ and $\delta>0$ with $\phi(\lambda,\delta)>0$. Then, for  every $\epsilon>0$  there exists a finite set $A\subset \bZ^d$ such that
\begin{align}\label{eq:prop1.1_1}
 &   \bP_{\lambda,\delta} 
\left( \leftidx{_{A}}A_1^o \supset A \right) >0;
\\ &  \bP_{\lambda,\delta} \left( A_t^A = \emptyset \text{ for some } t>0 \right) < \frac{\epsilon^2}{2}, \label{eq:prop1.1_2}
    \end{align}
    and so that, for any $N\in \bN$, there exists sequences $(L_j)$ and $(T_j)$ with $L_j=L_j(N,A)$ and $T_j=T_j(N,A)$, both diverging to infinity and satisfying, for each $j$,  
\begin{equation}\label{eq:Prop1.2}
\bP_{\lambda,\delta} \left( \left|  \leftidx{_{L_j}}A_{T_j}^{A} \right| > N \right) = 1-\epsilon.
\end{equation}
    Moreover, if $\lambda$ is symmetric, then the set $A$ can be chosen to be symmetric too, in the sense that, for any $x \in A$ we also have that $A$ contains any $\tilde{x}$ obtained by either changing the signs of any of the $d$ components of $x=(x_1,\dots,x_d)$ or after permuting them.    \end{proposition}
    
     In the case where $\lambda \in \Lambda$ is \emph{irreducible}, that is, the set $\{ y \in \bZ^d \colon \lambda_{o,y} > 0 \}  \text{ spans }\bZ^d $, 
 this follows as in the proofs in \cite{LiggettSIS1999} after minor modifications only. For the more general case covered here we in addition apply arguments as in \cite[Lemma 3.14]{bezuidenhout1994critical}. See Subsection \ref{sec:proofProp1.1} for the details.

The next proposition, inspired by \cite[Section 7]{grimmett2002directed}, is a modification of  \cite[Proposition I.2.8]{LiggettSIS1999}. 
To state it precisely, for any $L,T\in [0,\infty)$ and $A \subset B_L$, denote by $I_{L,T}^A \subset B_{L,T}$ the \emph{infected region} obtained by using the graphical representation within the space-time region $B_{L,T}$ only, that is,
\begin{equation}\label{eq:infectedSet}
I_{L,T}^A  \coloneqq \{(x,s)\in B_{L,T} \colon (x,s) \in  \leftidx{_{L}}A_s^A  \}. 
\end{equation}
Further, consider the random variable $E_{L,T,R}^{A} = E_{L,T,R}^{A}(I_{L,T}^A)$ that outputs the mean (or expected) number of infection arrows from $I_{L,T}^A$ to the space-time region  $S_{L,T,R}$.   That is,
\begin{equation}\label{eq:meanNumbInfect}
    E_{L,T,R}^{A} \coloneqq \sum_{x \in B_L} \sum_{y \in S_{L,T,R}} \lambda_{x,y} \int_0^T \ind_{\{(x,s) \in I_{L,T}^A\}} ds. 
\end{equation}
 For the case $R=\infty$ we  simply write $E_{L,T}^{A}$. 
 Note that, by definition, $E_{L,T,R}^{A}$ is a function of $I_{L,T}^A$ and hence measurable with respect to $\cF_{L,T} \coloneqq \cF_{B_{L,T}}$. 

\begin{proposition}\label{Prop2}
Let $\lambda \in \Lambda$ and $\delta>0$. Then, for any two increasing sequences $(L_j)$ and $(T_j )$, both diverging to infinity, and for any $M, N > 0$ and $A \subset \bZ^d$ finite, 
 \begin{equation}\label{eq:Prop2_key}
       {\lim\sup}_{j \rightarrow \infty} \bP_{\lambda,\delta}(E_{L_j, T_j}^{A}  \leq M) \bP_{\lambda,\delta}(|\leftidx{_{L_j}}A^{A}_{T_j}| \leq N) \leq \bP_{\lambda,\delta} (\exists s \colon A_s^A = \emptyset).
   \end{equation}
   \end{proposition}

Note that \cite[Proposition I.2.8]{LiggettSIS1999} gives the exact same statement as Proposition \ref{Prop2}, but with $E_{L_j, T_j}^{A}$ replaced by the  maximal number of  infected points on the boundary of $B_{L_j,T_j}$ and satisfying certain additional properties. 
 It is this modification, together with the tail regularity assumption,  both first proposed in \cite[Section 7]{grimmett2002directed} for a related long-range oriented percolation model, 
 that enable us to treat the case of  long-range interactions in the symmetric setting. 
  
 The  proof of Proposition \ref{Prop2} follows the same lines of reasoning as that of  \cite[Proposition I.2.8]{LiggettSIS1999} and is presented in Subsection \ref{sec:proofProp2}. There we also present the proof of the following inequality. 
 
\begin{proposition}\label{prop:tailbound}
Assume that $\lambda \in \Lambda$ satisfies  $\lambda_{o,y} \simeq \|y\|^{-\alpha}$ for some $\alpha>d$. Then there are $r \in (1,\infty)$, $\xi \in (0,1)$ and $L^*\in \bN$ so that, for any $L>L^*$, $T>0$  and $A \subset B_L$,
    \begin{equation}
        E_{L,T}^{A}\leq \frac{1}{1-\xi} E_{L,T,2r L}^{A}.
    \end{equation}
\end{proposition}

We are now set for the proof of Theorem \ref{thm:fstc} for which we will combine the just stated propositions together with basic properties of the LRCP as presented in Section \ref{sec:preliminiaries}.

\begin{proof}[Proof of Theorem \ref{thm:fstc}]
Let $\lambda \in \Lambda$ be symmetric and satisfy  $\lambda_{o,y} \simeq \|y\|^{-\alpha}$ for some $\alpha>d$, and let $\delta>0$ for which $\phi(\lambda,\delta)>0$. Fix  $r>1$, $\xi \in (0,1)$ and $L^* \in \bN$ such that the statement of Proposition \ref{prop:tailbound} holds.  Further, fix $\epsilon>0$, and consider $\tilde{\epsilon}>0$ which will be a function of $\epsilon$ that we specify at the end of the proof.
\begin{enumerate}
\item Fix $A\subset \bZ^d$, finite  and symmetric and so that Proposition \ref{prop1.1} holds with $\epsilon$ replaced by $\tilde{\epsilon}$. 
Further, set $\gamma \coloneqq \bP_{\lambda,\delta}( \leftidx{_{A }}|A_1^o \supset A)>0$. 

\item Let $N' = N'(\gamma)$ satisfy $(1-\gamma)^{N'} \leq \tilde{\epsilon}$. Further, choose $N=N(n,N')$ so large that any $\Delta \subset \bZ^{d}$ with $|\Delta|=N$ 
contains a subset of size $N'$ where all points are at distance at least $2n+1$, with  $n= \min (m \in \bN \colon A\subset B_m)$.  The reason for the choice of parameters here will be apparent in Step 5 below.

\item Choose $K=K(N,n)$ so large that any $\Delta \subset \bZ^{d+1}$ with $|\Delta|=K$ 
contains a subset of size $N' d2^{d}$ where all points are at distance at least $2n+2$. By possibly taking  $K$ even larger, we also assume that $\exp(-2\sqrt{K}) \leq \tilde{\epsilon}/2$ and that $4 \leq K^{1/4}$.
The reason for the choice of parameter here will be apparent in Step 6  below. 

\item Fix a sequence $(L_j,T_j)_{j \geq 1}$ with $\lim_{j \rightarrow \infty} L_j=\infty$ and $ \lim_{j \rightarrow \infty} T_j=\infty$, possibly depending on $N, K$ and $A$, satisfying  for each $j\geq 1$,
\begin{equation} L_jT_j \geq 4K, \quad L_j \geq \max(n, L^*), \quad T_j>1,\end{equation} and also that
\begin{equation}\label{eq:keyEstimate}
\bP_{\lambda,\delta} \left( \left| \leftidx{_{L_j}}A_{T_j}^{A} \right| >  2^d N\right) = 1-\tilde{\epsilon}.
\end{equation}
That this is possible is immediate  from Proposition \ref{prop1.1}. 
The reason why we require an equality in \eqref{eq:keyEstimate} will be apparent in Step 6 below.

\item Utilising the properties obtained in the previous Steps, we argue next that,  for any $j\in \bN$, 
\begin{align}\label{eq:block1}
 \bP_{\lambda,\delta}\left(\exists x\in B_{L_j}^+  \colon A(x) \subset \leftidx{_{{L_j}+2n}}{A}{_{T_j+1}^{A}}  \right)
 \geq (1-\tilde{\epsilon}^{2^{-d}})(1-\tilde{\epsilon}).
\end{align}
Indeed, for $j \in \bN$ fixed, since the LRCP is positively associated and the infection parameter $\lambda$ is translation invariant and symmetric, it follows by \eqref{eq:keyEstimate} and exactly the same argument as that for \cite[Proposition I.2.6]{LiggettSIS1999} that 
\begin{equation}\label{eq:block1.1}
\bP_{\lambda,\delta} \left(\left| \leftidx{_{L_j+2n}} A_{T_j}^{A} \cap B_{L_j}^+ \right| >  N\right) \geq 1 -\tilde{\epsilon}^{2^{-d}},
\end{equation}
using also that $A$ is symmetric.  In this event, by how we fixed $N$ in Step 2, there are at least $N'$ points within $_{L_j+2n} A_{T_j}^{A} \cap B_{{L_j}}^+$, all of which are at distance at least $2n+1$ apart.  
Further, by the independence of the Poisson processes within disjoint space-time regions in the graphical construction, for any set of points $x \in \bZ^d$ which are all at least distance $2n+1$ apart (with respect to the sup-norm), the events $\{A(x) \subset \leftidx{_{A(x)}} \eta_{1}^{x} \}$ are all mutually independent. Moreover, by translation invariance, they all have equal probability of at least $\gamma$ to occur. Now, recall that  by our choice of $N'$ in Step 2 we have that $(1-\gamma)^{N'} \leq \tilde{\epsilon}$. Therefore, again by the independence govern by the graphical representation and translation invariance, on the event within the probability in \eqref{eq:block1.1}, 
we conclude that the event
\begin{align}
 \left\{ \exists x\in B_{L_j}^+  \colon A(x) \subset \leftidx{_{L_j+2n}}{A}{_{T_j+1}^{A}} \right\}
 \end{align}
 occurs with probability at least $1-\tilde{\epsilon}$ and  from  which, when combined with \eqref{eq:block1.1}, we conclude \eqref{eq:block1}.

\item The following is the key step of the argument to cover the case of long-range interactions and is where we apply Propositions \ref{Prop2} and \ref{prop:tailbound}. 
Particularly, for $L,T,R \in [0,\infty)$,  let $M_{L,T,R}^A$ denote the maximal number of points within $S_{L,T,R}$ to which there is an infection arrow from a space-time point in $I_{L,T}^A$ and  complying with the following: any two distinct points  $(x,s)$ and $(y,t)$ contained in $M_{L,T,R}^{A}$  satisfies either that $|s-t|> 1$ or that $\|x-y\|_{\infty} \geq 2n+1$.

 We claim that there exists $L,T \in [0,\infty)$ contained in the sequence $(L_j,T_j)$ constructed in Step 4 such that 
\begin{equation}\label{eq:keyEstimate1111}
    \bP_{\lambda,\delta} \left( M_{L,T,2r L}^{A}> N' d2^d \right) > 1-\tilde{\epsilon}.
\end{equation}

 In the following we provide the detailed argument to why \eqref{eq:keyEstimate1111} holds.  Recall from \eqref{eq:meanNumbInfect} that  $E_{L,T}^{A}$ gives the 
 expected number of infection arrows from $I_{L,T}^{A}$ to the space-time region  $(\bZ^d \setminus B_L)\times [0,T]$, where $I_{L,T}^{A}$ is given by \eqref{eq:infectedSet}.
 Then, by our choice of $A \subset \bZ^d$ in Step 1 and our choice of $(L_j,T_j)$ in Step 4, \eqref{eq:Prop2_key} implies that
\begin{equation}
       {\lim\sup}_{j \rightarrow \infty} \bP_{\lambda,\delta} \left(E_{L_j, T_j}^{A}  \leq \frac{16}{1-\xi} K \lambda_{\infty} \right) \tilde{\epsilon} \leq \frac{\tilde{\epsilon}^2}{2},
\end{equation}
where $\lambda_{\infty}$ is as in \eqref{eq:lambda_infty}. 
Therefore, we have that 
\begin{equation}
    \bP_{\lambda,\delta} \left( E_{L,T}^{A}>\frac{16}{1-\xi} K \lambda_{\infty} \right) > 1-\tilde{\epsilon}/2
\end{equation}
for some $(L,T)$ within the sequence $(L_j,T_j)$, and so, by Proposition \ref{prop:tailbound}, we also have that
\begin{equation}\label{eq:claimBlokkis}
    \bP_{\lambda,\delta} \left( E_{L,T,2r L}^{A}> 16 K\lambda_{\infty} \right) > 1-\tilde{\epsilon}/2.
\end{equation}
Now, let $Y$ denote the  actual number of infection arrows from $I_{L,T}^{A}$  to $S_{L,T,2r L}$. Note that, given $I_{L,T}^{A}$ and therefore also $E_{L,T,2r L}^{A}$, the random variable $Y$ is Poisson distributed with mean $\mu_Y=E_{L,T,2r L}^{A}$.  
Further, for $x \in S_{L,T,2r L}$ and  $i \in \{1, \dots,\lfloor T \rfloor\}$, let  $Y_{x,i}$ denote the number of infection arrows from $I_{L,T}^{A}$  to the space time region $\left\{x\right\} \times \left[T \frac{(i-1)}{\lfloor T \rfloor}, T\frac{i}{ \lfloor T \rfloor}\right)$.
Then, given $I_{L,T}^{A}$, we have that
\begin{equation}
Y = \sum_{x \in B_{L+R}\setminus B_L} \sum _{i=1}^{\lfloor T \rfloor} Y_{x,i},
\end{equation}
where, moreover, the $Y_{x,i}$'s  are by construction themselves all independent $Poi(\mu_{x,i})$-random variables satisfying that 
\begin{equation}\label{eq:boundOnMu}
\mu_Y =  \sum_{x \in B_{L+R}\setminus B_L} \sum _{i=1}^{\lfloor T \rfloor}  \mu_{x,i} \quad \text{ and }  \quad \sup_{(x,i)} \mu_{x,i} \leq 2 \lambda_{\infty} 
\end{equation}
Here the constant $2$ in the latter term has simply  been chosen such that it is greater than $T/\lfloor T \rfloor$ for any $T>1$. 

We claim now that, conditioned on the event $\{E_{L,T,(1+2r)L}^{A}> 16 K \lambda_{\infty} \}$, the event
\begin{equation}\label{eq:claimBlokkis2}
\left\{ \sum_{\sum_{x \in B_{L+R}\setminus B_L}}\sum _{i=1}^{\lfloor T \rfloor} \ind_{\{Y_{x,i} \geq 1\}} \geq K \right\}
\end{equation}
occurs with probability at least $1-\tilde{\epsilon}/2$. 
To see this, partition the sets of points 
\begin{equation}
\{(x,i) \colon x  \in B_{(2r+1)L}\setminus B_L , i = 1,\dots, \lfloor T \rfloor \}
\end{equation}
 into sets $\{\Delta_j \colon j =1,\dots,l\}$ such that $\sum_{(x,i)\in \Delta_j} \mu_{x,i} \geq \ln(2)$ for each $j=1,\dots,l$.
Since \eqref{eq:boundOnMu} holds with $\mu_Y \geq 16 K\lambda_{\infty}$, we can do this in such a way that  $l\geq 4K$. Let, for $j =1,\dots, l$, 
\begin{equation} Z_j = \left\{\begin{array}{cc}1 & \text{ if } \sum_{(x,i)\in \Delta_j} Y_{x,i} \geq 1; \\0 & \text{ otherwise. }\end{array}\right.
\end{equation}
Note that, by construction, the random variables $(Z_j)$ are independent conditional on $I_{L,T}^A$. Moreover, for each $j=1,\dots,l$, we have that 
\begin{equation}
\bP \left(Z_j =1 \right) = \bP\left(\sum_{(x,i)\in \Delta_j} Y_{x,i} \geq 1\right) \geq 1/2,
\end{equation} where the inequality holds since $\sum_{(x,i)\in \Delta_j} Y_{x,i}$ is Poissonian with mean larger than $\ln(2)$.  Therefore, by Hoeffding's inequality \cite[Theorem 1]{Hoeffding1963}, we have that
\begin{equation}
\bP_{\lambda,\delta} \left( \sum_{i=1}^l  Z_i  \geq l/2- l^{3/4}  \right) \geq 1-\exp(-2 \sqrt{l} ).
\end{equation}
From this, since $l \geq 4K$, and by our assumptions on $K$ in Step 3, the inequality \eqref{eq:claimBlokkis2} follows. 

The claim \eqref{eq:keyEstimate1111} holds as a consequence of our choice of $K$ in Step 3 and the estimates in \eqref{eq:claimBlokkis} and \eqref{eq:claimBlokkis2}. Indeed, by combining these estimates it follows that the event \eqref{eq:claimBlokkis2} holds with a probability  of at least $1-\tilde{\epsilon}$. In the case of that event, by our choice of $K$ in Step 3, there necessarily is a  subset of the $Y_{x,i}$'s of size $N'd2^d$, all of which are strictly positive and such that the corresponding labelings are all at least of distance $2n+2$ apart. This implies that $M_{L,T,R}^m>N'd2^d$, and proves the claim. 

\item Similarly to the argument in Step 5, by further utilising the bound \eqref{eq:keyEstimate1111} obtained in Step 6, we argue that 
\begin{align}\label{eq:block2}
   \P_{\lambda,\delta}\left(\exists (x,t) \in S_{L,T,2r L}^{i,\pm} \colon A(x) \subset \leftidx{_{(1+2r)L}}{A}{_{t+1}^{A}}   \right)
 \geq (1-\tilde{\epsilon}^{2^{-d}/d})(1-\tilde{\epsilon}),
\end{align}
for each $i =1,\dots,d$. Indeed, since the LRCP is positively associated, and both $A$ and the infection parameter $\lambda$ are symmetric, it follows by
 \eqref{eq:keyEstimate1111} and exactly the same argument as that for \cite[Proposition I.2.11]{LiggettSIS1999}, also using that $\lambda$ is translation invariant, that 
\begin{equation}\label{eq:block2.1}
\bP_{\lambda,\delta} \left( M_{L,T,2r L}^{A}(S_{L,T,2r L}^{i,\pm}) > N' \right) \geq 1-\tilde{\epsilon}^{2^{-d}/d}
\end{equation}  
for each $i =1,\dots,d$. In the event that this happens for a particular $i=1,\dots,d$, in the same vein as we argued in Step 5, it follows that  
\begin{align}
 \bP_{\lambda,\delta}&\left( \exists (x,t)  \in S_{L,T,2r L}^{i,\pm} \colon A(x) \subset  \leftidx{_{2(1+r)L}}{A}{_{t+1}^{A}} \right) \geq 1-\tilde{\epsilon},
 \end{align}
  using again that the process is positively associated,  that  $L\geq n$, $(1-\epsilon)^{N'} \leq \tilde{\epsilon}$, translation invariance and that the Poisson process of the graphical representation within disjoint space-time regions are independent. 
The  bound \eqref{eq:block2} therefore follows from  this latter inequality and \eqref{eq:block2.1}.
\end{enumerate}

We now summarize the proof of the theorem which follows by the estimates  \eqref{eq:block1} and \eqref{eq:block2} above, and by appropriately tuning of the parameters introduced. 
Indeed, set $\tilde{\epsilon}$ such that $1-\epsilon \geq  (1-\tilde{\epsilon}^{2^{-d}/d})(1-\tilde{\epsilon})$ and $1-\epsilon \geq (1-\tilde{\epsilon}^{2^{-d}})(1-\tilde{\epsilon})$ and let  $A=A(\tilde{\epsilon}) \subset \bZ^d$ be as in Step 1 above. 
Further, fix $L',T' \in (0,\infty)$ as in Step 6 above so that \eqref{eq:keyEstimate1111} is satisfied. Then the inequalities \eqref{eq:block1} and \eqref{eq:block2}  holds with probability $1-\epsilon$ and parameters $A$, $L'$ and $T'$, from which the claim of Theorem \ref{thm:fstc} with parameters $A$, $L =L'$ and $T=T'+1$ directly follows.
\end{proof}

\subsection{Proof of Proposition \ref{prop1.1}}\label{sec:proofProp1.1} 

The first part of Proposition \ref{prop1.1} follows as in \cite[Lemma 3.14]{bezuidenhout1994critical}, as presented next.

\begin{lemma}\label{lem1:prop1}
 Let $\lambda \in \Lambda$  and $\delta>0$ with $\phi(\lambda,\delta)>0$. Then, for any  $\epsilon>0$ there exists a finite set $A\subset \bZ^d$ such that
\begin{align}
 &   \bP_{\lambda,\delta} \left( \leftidx{_{A}}A_1^o \supset A \right) >0;\label{eq3:lemma_prop1.1}
 \\ &\bP_{\lambda,\delta} \left( A_t^A \neq \emptyset \text{ for some } t>0 \right) > 1- \frac{\epsilon^2}{2}. \label{eq1:lemma_prop1.1}
    \end{align}
    If $\lambda$ is symmetric, then the set $A$ can be chosen to be symmetric too.
    \end{lemma}
    
    \begin{proof}
    This follows by the same argument as in the proof of \cite[Lemma 3.14]{bezuidenhout1994critical}, but for completeness we provide the details. 
      Firstly, we claim that there is a $A\subset \bZ^d$ finite such that
    \begin{equation}\label{eq2:lemma_prop1.1}
    \bP_{\lambda,\delta} \left(A \subset A_1^{o} \right)>0 
    \end{equation}
    and such that the inequality in \eqref{eq1:lemma_prop1.1} holds. Indeed, suppose not, then for each $A\subset \bZ^d$ finite satisfying \eqref{eq2:lemma_prop1.1} there is necessarily a finite time, say $\tau_A$, such that 
    $\bP_{\lambda,\delta}( A_{\tau_A}^A = \emptyset) \geq \epsilon^2/2$. Then, combining this with the strong Markov property of the LRCP iteratively a number of times, it follows that 
    \begin{equation}
    \bP_{\lambda,\delta}(A_t^{o} \neq \emptyset \text{ for some } t>0 ) =0, 
    \end{equation} 
    which contradicts the assumption that $\phi(\lambda,\delta)>0$. Thus, there necessarily exists a set $A\subset \bZ^d$ such that \eqref{eq1:lemma_prop1.1} and \eqref{eq2:lemma_prop1.1} holds. 
      Now, by continuity of probability measures, we therefore have that 
    \begin{equation}
    \bP_{\lambda,\delta} \left(A \subset \leftidx{_n}A_1^{o} \right) >0
    \end{equation}
for some $n \in \bN$.    Consequently, there must be a set $C \subset B_n$ such that 
       \begin{equation}
    \bP_{\lambda,\delta} \left(A \subset \leftidx{_{C}}A_1^{o} \right) >0
    \end{equation} and such that $\bP_{\lambda,\delta}  (x\in A_1^{o})>0$ for each $x\in C$ holds. By this we conclude \eqref{eq3:lemma_prop1.1} by possibly enlarging $A$ to equal $C$. Lastly, if $\lambda$ is symmetric, we may assume that $C$, and, hence, $A$ is symmetric too. This completes the proof.
    \end{proof}

The next lemma corresponds to  \cite[Proposition 2.2]{LiggettSIS1999}. 

\begin{lemma} 
\label{lem:prophelp}
Let $\lambda \in \Lambda$ and $\delta>0$. Then, for every finite $A \subset \bZ^d$ and every $N \geq 1$,
\begin{equation}
    \lim_{t \rightarrow \infty} \lim_{L \rightarrow \infty} \bP_{\lambda,\delta} \left(\left|  \leftidx{_L}A_t^A \right| \geq N\right) =  \bP_{\lambda,\delta} \left(\forall t>0, A_t^A \neq \emptyset \right)
\end{equation}
\end{lemma}

\begin{proof}
    Firstly, by continuity of probability measures, we have that 
    \begin{equation}\label{eq:helpMeFirst}
        \lim_{L \rightarrow \infty} \bP_{\lambda,\delta}(| \leftidx{_L}A_t^A| \geq N) = \bP_{\lambda,\delta}(|A_t^A| \geq N).
    \end{equation}
    Next, we argue that 
    \begin{equation}\label{eq:helpMe}
         \bP_{\lambda,\delta}(A_t^A = \emptyset \text{ for some } t | \cF_s) \geq \left[ \frac{1}{1+\lambda_{\infty}|A_s^A|}\right]^{|A_s^A|}.
    \end{equation}
    Indeed, utilising the graphical construction, for any $A\subset \bZ^d$ finite, we have that 
\begin{align}
     &\bP_{\lambda,\delta} \left(A_t^A = \emptyset \text{ for some } t>0 \right) \\
    & \geq \bP_{\lambda,\delta} \left(\cap_{x \in A} \{ \min (s \in H_{x} \colon s\geq 0) \leq \min( s \in I_{y,x} \colon s\geq 0, y \neq x) \} \right)\\
    & = \prod_{x \in A}  \bP_{\lambda,\delta} \left(\min (s \in H_{x} \colon s\geq 0) \leq \min( s \in I_{y,x} \colon s\geq 0, y \neq x)  \right)\\
    & = \left[ \frac{1}{1+\lambda_{\infty}}\right]^{|A|}.
\end{align}
 Further, denoting by $X$ the indicator function
 \begin{equation} 
 X= \left\{\begin{array}{cc}1 & \text{ if }A_t^A = \emptyset \text{ for some } t>0; \\0 & \text{ otherwise, }\end{array}\right.
\end{equation}
 we have that 
\begin{equation}\label{eq:helpME2}
    \bP_{\lambda,\delta} (X=1 \mid \cF_s) \rightarrow X \quad a.s.\
\end{equation}
as $s\rightarrow \infty$ since $(\bE(X | \cF_s))$ forms a Doobs-martingale. Therefore, it holds that
\begin{equation}\label{eq:lastHelpMe}
    \bP_{\lambda,\delta}(\lim_{t \rightarrow \infty} |A_t^A| = \infty \mid  X=1 )=1.
\end{equation}
Indeed, if not, then there is an $N \in\bN$ such that $|A_t^A| \leq N$ occurs with positive probability for all $t$ large. However,  by  \eqref{eq:helpMe}, this implies that \begin{equation}
\bP_{\lambda,\delta}(A_t^A = \emptyset \text{ for some } t)>0,
\end{equation} which stands in contradiction to \eqref{eq:helpME2}. Lastly, from \eqref{eq:lastHelpMe} and using \eqref{eq:helpMeFirst}, we conclude the statement of the lemma since $\lim_{t \rightarrow \infty} \bP_{\lambda,\delta}( |A_t^A| \geq N) =\bP_{\lambda,\delta}(X=1)$.
\end{proof}

\begin{proof}[Proof of Proposition \ref{prop1.1}]
This follows by the previous lemma by an argument similar to the first part of the proof of Theorem 2.12 in \cite{LiggettSIS1999} (to the conclusion of Equation (2.16) therein). 
 Indeed, given $\epsilon>0$, by Lemma \ref{lem1:prop1}, there exists $A\subset \bZ^d$ finite and symmetric such that \eqref{eq:prop1.1_1} and \eqref{eq:prop1.1_2} hold, and where $A$ can be chosen symmetric if $\lambda$ is. 
Moreover, by Lemma \ref{lem:prophelp} and since $0 < 1-\epsilon < 1- \epsilon^2$, for all $t>0$ large we have that
\begin{equation}\label{eq:Prop1help1}
    \lim_{L \rightarrow \infty} \bP_{\lambda,\delta} \left( \left| \leftidx{_L}A_t^A \right|\geq N \right) > 1- \epsilon.
\end{equation}
Thus, by continuity of probability measures, for each such a $t>0$ we can find an $L$ such that also 
\begin{equation}\label{eq:Prop1help2}
\bP_{\lambda,\delta} \left(| \leftidx{_L}A_t^A|\geq N\right) > 1-\epsilon.
\end{equation}
The claim now follows by choosing an increasing sequence of $T_j$'s satisfying \eqref{eq:Prop1help1} and there-next choosing a corresponding sequence $L_j$, increasing in $j$, satisfying \eqref{eq:Prop1help2}. 
Note that we may here choose $(L_j,T_j)$ as large as we like. 
Further, for such a fixed sequence $(L_j,T_j)$, we may increase $T_j$ to obtain an equality in \eqref{eq:Prop1help2}, i.e., to guarantee that \eqref{eq:Prop1.2} holds. This is so since the function $f(t)= \bP_{\lambda,\delta} (|\leftidx{_L}A_{T_j+t}^A|\geq N)$ is continuous in $t$ (as e.g.\ follows by graphical representation of the LRCP, see  the discussion in \cite[Page 39]{LiggettSIS1999}) and since $f(t)$ converges to $0$ as $t\rightarrow \infty$. 
\end{proof}

We next show that for $\lambda \in \Lambda$ symmetric and irreducible, the set $A$ can be set to equal $B_n$ for some $n\in \bN$. 

\begin{proposition}\label{Prop3}
Let $\lambda \in \Lambda$ be symmetric and irreducible, and let $\delta>0$. Then there exists $N=N(\lambda)\geq 1$ such that, for all $n\geq N$, 
\begin{equation}\label{eq:Prop3}
\bP_{\lambda,\delta}\left( \leftidx{_{[0,n]^d}}A_1^o \supset [0,n]^d\right) >0.
\end{equation}
\end{proposition}

\begin{proof}[Proof of Proposition \ref{Prop3}]
This following by applying \cite[Lemma 2.1]{baumler2023continuity}, which gives a similar statement for a long-range percolation models. Therein two nodes $x,y\in \bZ^d$ with $x\neq y$ are connected by an open edge independently with probability $p_{x,y}=\beta J_{x,y}$ where $J_{x,y}$ is assumed to satisfy the Conditions a)-c) and f) and $\beta>0$.

We say that $o$ is connected to $y \in [0,N]^d$ if there exists a finite sequence $o=x_0,x_1,\dots, x_m = y$ so that $\lambda_{x_i,x_{i+1}}>0$ for all $i=0,\dots, m-1$. The vertices are connected within $\{0,\dots,m\}^d$ 
if all the vertices $x_0,\dots,x_m$ in the sequence are contained in this set. 
Then, by \cite[Lemma 2.1]{baumler2023continuity},  for all large $N$ and every $y \in [0,N]^d$, the vertex $o$ is connected to $y$  within $[0,N]^d$.  
Indeed, letting $p_{x,y}= \lambda_{x,y}/\lambda_{\infty}$,  \cite[Lemma 2.1]{baumler2023continuity} implies that such a sequence exists along which $p_{x_i,x_{i+1}}>0$ for $i=1,\dots,n$. 
Thus, there necessarily has to exist a path from $o$ to $x$ in $[0,N]^d$ along which $\lambda$ is positive.  

Now, note that, if $o$ and $y$ are connected within $[0,N]^d$, then with positive probability there exists an infection path in the graphical construction within the time window $[0,1]$ from $o$ to $y$ fully contained in $[0,N]^d$. 
Moreover, with positive probability there are no healing events in $[0,N]^d$  within the time window $[0,1]$. 
In particular, we have that, 
\begin{equation}\label{eq:BoundIrr}
    \bP_{\lambda,\delta} \left( \leftidx{_{[0,N]^d}}A_1^x \supset \{y\} \right) >0.
\end{equation}
From this, and since  the LRCP is positively associated, the conclusion of the proposition follows. 
\end{proof}

\subsection{Proofs of Proposition \ref{Prop2} and Proposition \ref{prop:tailbound}}\label{sec:proofProp2} 
  Before we provide the full proof of Proposition \ref{Prop2}, we give the following lemma, which corresponds to the first step of the proof of \cite[Proposition 2.8]{LiggettSIS1999}. For this, for $L,T>0$ and $A \subset B_L$,  recall the definition of $E_{L,T,R}^{A}$ (and $E_{L,T}^{A}$) from \eqref{eq:meanNumbInfect}, and consider the event
\begin{equation}
H_{L,T}^{A}(k) \coloneqq \left\{ E_{L,T}^A + | \leftidx{_L}A_T^A | \leq k \right\}, \quad k \in \bN. 
\end{equation}

\begin{lemma}\label{lem:Prop2}
Let $\lambda \in \Lambda$ and $\delta>0$. Then, for any $L,T>0$ and $A \subset B_L$, a.s.\ on $H_{L,T}^{A}(k)$, $k \in \bN$, it holds that
\begin{equation} 
\bP_{\lambda,\delta} \left( \exists s>0 \colon A_s^A = \emptyset  \mid \cF_{L,T} \right) \geq \left(e(1+\lambda_{\infty})\right)^{-k}.
\end{equation}
\end{lemma}

\begin{proof}
First note that with probability $1/(1+\lambda_{\infty})$ the next event occurring at a spatial location $x \in\leftidx{_L}A_T^A$ is a healing event. Therefore, since this happens independently for each such point, if $| \leftidx{_L}A_T^A | =l$ for some $l\in \bN$, then the conditional probability that no $x \in \leftidx{_L}A_T^A$ contributes to survival is at least $(1+\lambda_{\infty})^{-l}$. 
Further, since the number of infection arrows from $I_{L,T}^A$ is Poisson distributed with mean $E_{L,T}^A$, conditioned on $E_{L,T}^A = j \in [0,\infty)$, the probability that there are no infection arrows from any point in $I_{L,T}^A$ is simply $e^{-j}$. Combining these two observations yields the statement of the lemma.
\end{proof}

\begin{proof}[Proof of Proposition \ref{Prop2}]
With Lemma \ref{lem:Prop2} at hand, the claim of Proposition \ref{Prop2} follows almost exactly as in the second part of the proof of \cite[Proposition 2.8]{LiggettSIS1999}.  Firstly, by the martingale convergence theorem, it holds that
\begin{equation}
\lim_{j \rightarrow \infty} \bP_{\lambda,\delta} \left(  \exists s>0 \colon A_s^A = \emptyset  \mid \cF_{L_j,T_j} \right) = \ind_{\{\exists s>0 \colon A_s^A = \emptyset \}}, \quad a.s.
\end{equation}
By Lemma \ref{lem:Prop2}, for any $l,k \in \bN$, the probability $\bP_{\lambda,\delta}(  \exists s>0 \colon A_s^A = \emptyset  \mid \cF_{L_j,T_j}) $ is bounded below on 
$H_{L_j,T_j}^{A}(l+k)$ and therefore
\begin{equation}
\left\{ H_{L_j,T_j}^{A}(l+k) \text{ i.o. } \right\} \subset \left\{\exists s>0 \colon A_s^A = \emptyset \right\},
\end{equation}
which in turn implies that
\begin{equation}
\limsup_{j\rightarrow \infty} \bP_{\lambda,\delta} \left(H_{L_j,T_j}^{A}(l+k) \right) \leq \bP_{\lambda,\delta} \left(\exists s>0 \colon A_s^A = \emptyset\right).
\end{equation}
By this we conclude the proof by noting that, since the events $\{E_{L_j,T_j}^A \leq l\}$ and $\{  | \leftidx{_{L_j}}A_{T_j}^A | \leq k\}$ are decreasing, we have
that
\begin{align}
\bP_{\lambda,\delta}\left(H_{L_j,T_j}^{A}(l+k)\right) &\geq \bP_{\lambda,\delta}\left(E_{L_j,T_j}^A\leq l,  | \leftidx{_{L_j}}A_{T_j}^A | \leq k\right) \\ &
\geq   \bP_{\lambda,\delta}\left(E_{L_j,T_j}^A \leq l\right) \cdot \bP_{\lambda,\delta}\left(  | \leftidx{_{L_j}}A_{T_j}^A | \leq k\right),
\end{align}
using that the LRCP is positively associated.
\end{proof}

\begin{proof}[Proof of Proposition \ref{prop:tailbound}]
Assume that $\lambda \in \Lambda$ satisfies  $\lambda_{o,y} \simeq \|y\|^{-\alpha}$ for some $\alpha>d$. 
Then, for some $r > 1$ and $\xi \in (0,1)$ there is an $L^*\in \bN$ so that, for all $L>L^*$,
  \begin{equation}\label{eq:tail_b}
  \sum_{y \colon \|y\|>r L }\lambda_{o,y} \leq \xi \sum_{y \colon \|y\|> L}\lambda_{o,y}   \end{equation}
 Indeed, for some constants $C_1,C_2<\infty$ and for any $r>1$ and $L\in (0,\infty)$ so that $rL \geq R$, we have that
 $ \sum_{y \colon \|y\|>r L }\lambda_{o,y} 
\leq  C_1 (r L)^{-\alpha + d}$
and 
 $ \sum_{y \colon \|y\|> L }\lambda_{o,y} 
\geq  C_2 L^{-\alpha + d}.$ 
 Thus, by tuning $r$ large so that $\xi=\frac{C_1 r^{-\alpha+d}}{C_2}<1$, we have that \eqref{eq:tail_b} holds for any $L \in \bN$. 
 
 Now, fix $L>L^*$, $T>0$  and $A \subset B_L$. Then we have that
\begin{align}
    E_{L,T}^{A}&= \sum_{x \in B_L}\sum_{y \in \bZ^d \setminus B_L} \lambda_{x,y}
 \int_0^T \ind_{\{(x,s) \in I_{L,T}^A\}} ds
\\ & = \sum_{x \in B_L}  \int_0^T \ind_{\{(x,s) \in I_{L,T}^A\}} ds \left( \sum_{y \in B_{(1+2r)L}\setminus B_L} \lambda_{x,y} + \sum_{y \in \bZ^d \setminus B_{(1+2r)L}} \lambda_{x,y}\right)
\\ \leq & \sum_{x \in B_L}  \int_0^T \ind_{\{(x,s) \in I_{L,T}^A\}} ds \left( \sum_{y \in B_{(1+2r)L}\setminus B_L}  \lambda_{x,y} + \sum_{y \colon \|y-x\| \geq 2r L} \lambda_{x,y}\right)
\\ \leq & \sum_{x \in B_L}  \int_0^T \ind_{\{(x,s) \in I_{L,T}^A\}} ds \left(\sum_{y \in B_{(1+2r)L}\setminus B_L}  \lambda_{x,y} +  \xi \sum_{y \colon \|y-x\| \geq 2L} \lambda_{x,y}\right)
\\ \leq & \sum_{x \in B_L}  \int_0^T \ind_{\{(x,s) \in I_{L,T}^A\}} ds \left( \sum_{y \in B_{(1+2r)L}\setminus B_L}  \lambda_{x,y} +  \xi \sum_{y \in \bZ^d \setminus B_L} \lambda_{x,y}\right)
\\ = & E_{L,T,2r L}^{A} + \xi E_{L,T}^{A},
\end{align} 
where we applied \eqref{eq:tail_b} in the second inequality. Thus, we find that 
\begin{equation}
E_{L,T}^{A} \leq \frac{1}{1-\xi} E_{L,T,2r L}^{A}
\end{equation}
and we therefore conclude the proof since this bound holds irrespectively of $T$, $A$ and $L> L^*$. 
\end{proof}

   

  \section{The non-symmetric case}\label{sec:non-sym}
  
  In this section we present how the proofs of Theorem \ref{thm:fstc} and Theorem \ref{thm:truncation}a) can be modified to cover the non-symmetric case  as stated in Theorem \ref{thm:truncation}b). 
   For this purpose we need to extend the definitions  given in the previous section  to  tilted boxes. That is,  for $L=(L_1,\dots,L_d)\in [0,\infty)^d$, $T \in [0,\infty)$ and $\theta=(\theta_1,\dots,\theta_d) \in \bR^d$, and writing $x=(x_1,\dots,x_d)$, we let
\begin{align}
&B_{L,T}^{\theta} \coloneqq \left\{ (x,t) \in \bZ^d \times [0,T) \colon | x_i- t \theta_i| \leq L_i,  i=1,\dots,d\right\}. 
\end{align}
Further, for $R=(R_1,\dots,R_d) \in [0,\infty)^d$, we denote by 
\begin{equation}
T_{L,T}^{\theta} \coloneqq \left\{ x \in \bZ^d \colon (x,T) \in B_{L,T}^{\theta}\right\} \text{ and }  S_{L,T,R}^{\theta} \coloneqq B_{L+R,T}^{\theta} \setminus B_{L,T}^{\theta}
\end{equation} 
 the top of $B_{L,T}^{\theta}$ and the $R$-boundary of $B_{L,T}^{\theta}$, respectively. 

In the following statement, and in the proof of Theorem \ref{thm:truncation}b), without loss of generality  we  assume that the interaction parameter is strictly positive in the $d$-th coordinate direction, i.e.\ that $\lambda_{o,y}>0$ for some $y \in \bZ^d$ with $y \cdot e_d>0$. With this assumption it is sufficient to consider the following set
\begin{equation}
S_{L,T,R}^{\theta, \pm} \coloneqq \left\{(x,t) \in S_{L,T,R}^{\theta} \colon \pm x_d \in [L_d+\theta_d t, L_d+R_d +\theta_d t]\right\},
\end{equation}
i.e.\ where the restriction on $x=(x_1,\dots,x_d)$ only concerns the d-th coordinate direction.

\begin{theorem}[Finite space-time condition for non-symmetric interactions]\label{thm:fstcWS}
Let $\lambda \in \Lambda$  satisfy  $\lambda_{o,y} =\mathcal{O}( \|y\|^{-\alpha})$ for some $\alpha>2d+1$,
 and let $\delta>0$ so that $\phi(\lambda,\delta)>0$. Then there is $r>1$ so that the following holds: for every  $\epsilon > 0$ there are choices of $A \subset \bZ^d$ finite and $L \in (0,\infty)^d$, $T\in (0,\infty)$ and $\theta =(0,\dots,0,\theta_d) \in \bR^d$ such that 
\begin{align}
\label{C3}\tag{C3}
 & \bP_{\lambda,\delta} \left(\exists x\in B_{L,T}^{\theta}  \colon A(x) \subset \leftidx{_{B_{2L,T+1}^{\theta}}}{A}{_{T+1}^{A}} \right) \geq 1-\epsilon;
 \\ &    \label{C4}\tag{C4}
      \bP_{\lambda,\delta} \left( \exists (x,t)  \in S_{L,T,2r LT}^{\theta,A,\pm} \colon A(x) \subset  \leftidx{_{B_{2(1+r)LT,T+1}^{\theta}}}{A}{_{t+1}^{A}} \right) \geq 1-\epsilon.
 \end{align}
\end{theorem}

Theorem \ref{thm:fstcWS} is similar to Theorem \ref{thm:fstc}, but without the symmetry assumption, and can be seen as an infinite-range version of \cite[Lemma 4.1]{bezuidenhout1994critical} when restricting the latter to the contact process with finite range interaction. Its proof is presented in the next subsection. 

\begin{proof}[Proof of Theorem \ref{thm:truncation}b)]
With Theorem \ref{thm:fstcWS} at hand, the proof of Theorem \ref{thm:truncation}b) goes along the same lines as that for  Theorem \ref{thm:truncation}a). Here one should note that Theorem \ref{thm:fstcWS} provides the key step for the proof of \cite[Proposition 5.1]{bezuidenhout1994critical}, which gives a statement akin of Proposition \ref{prop:renormalization1} that we now recall:  
for any  $\epsilon > 0$, there exists $A \subset \bZ^d$, $L \in [0,\infty)^d$ and $T \in [0,\infty)$   such that 
\begin{align}
 &\bP_{\lambda,\delta} \left(\exists (y,t) \in R(L,T)  \colon A(y)   \subset \leftidx{_{B_{L,T}^{A}}}{A}{_{t}^{A}} \right)> (1-\epsilon);
\\ &\bP_{\lambda,\delta} \left(\exists (y,t) \in L(L,T)  \colon A(y)  \subset \leftidx{_{B_{L,T}^{A}}}{A}{_{t}^{A}} \right)> (1-\epsilon),
 \end{align}
 where 
 \begin{align}
 &R(L,T)\coloneqq \left\{ (x,t) \in B_{L,T} \colon (x_d,t) \in [L_d/3,L_d] \times [T/3,T] \right\};
 \\ &L(L,T)\coloneqq \left\{ (x,t) \in B_{L,T} \colon (x_d,t) \in [-L_d,-L_d/3] \times [T/3,T] \right\}.
 \end{align}
 This latter inequality follows by arguing as in the proof of Proposition \ref{prop:renormalization1}, using the strong Markov property of the LRCP and the bounds of Theorem \ref{thm:fstcWS} a finite number of times to steer the spread of the infection to the designated space-time region, in a similar way as for the proof of Proposition \ref{prop:renormalization1}. In fact, by Theorem \ref{thm:fstcWS}, in doing so we may restrict to  the LRCP truncated at level $(1+2r)LT$ to which the proof of \cite[Proposition 5.1]{bezuidenhout1994critical} directly applies.  Furthermore, with these bounds we can apply the same renormalization argument as for the proof of Theorem \ref{thm:truncation}a) above, see also \cite[Section 6]{bezuidenhout1994critical} and particularly the proof of Theorem 2.4 therein. Since at this stage there are essentially no new arguments needed in order to adapt these argument to the current setting, we here omit  the detailed proof.
\end{proof}

\subsection{Proof of Theorem \ref{thm:fstcWS}}

In this subsection we detail how the proof of Theorem \ref{thm:fstc} can be extended to yield Theorem \ref{thm:fstcWS}. For this we also utilise ideas from \cite{bezuidenhout1994critical}, who covered the finite range case in a setting without assuming that the infection parameter is symmetric.  
But first we need to extend some definitions to more general space-time regions of $ \bZ^d \times[0,\infty)$. 

More specifically, for $L=(L_1,\dots,L_d)\in [0,\infty)^d$, $T \in [0,\infty)$,  $\theta=(\theta_1,\dots,\theta_d) \in \bR^d$ and $A \subset \bZ^d$ finite, similar to \eqref{eq:infectedSet} we denote by $I_{L,T}^{\theta,A}$ the "infected region" obtained by using the graphical representation within the box $B_{L,T}^{\theta}$ when initialising the process with the set $A$ infected, and similar to \eqref{eq:meanNumbInfect} by $E_{L,T,R}^{\theta,A}$  the expected number of infection arrows from $I_{L,T}^{\theta,A}$ to the space-time region $S_{L,T,R}^{\theta}$. Moreover, we let $E_{L,T,R}^{\theta,A,\pm}$ denote the  
expected number of infection arrows from $I_{L,T}^{\theta, A}$ to the space-time regions $S_{L,T,R}^{\theta,\pm}$, respectively, and consider the events
\begin{equation}
H_{L,T,R}^{\theta,A}(k) \coloneqq \left\{ E_{L,T,R}^{\theta,A,+} +  E_{L,T,R}^{\theta,A,-} + \left| \leftidx{_{B_{L,T}^{\theta}}}A_T^A \right| \leq k \right\}, \quad k \in \bN. 
\end{equation}
As in the previous section, if $R_i=\infty$ in each coordinate direction $i=1,\dots,d$, we suppress it from the notation. 

In addition to Proposition \ref{prop1.1},  the proof of Theorem \ref{thm:fstcWS} also rely on the following statement, which is an extension of \cite[Lemma 3.18]{bezuidenhout1994critical} to the long-range setting and assures that the LRCP spreads at an at most linear speed. 

\begin{proposition}\label{prop:fstcWS1}
Let $\lambda \in \Lambda$  satisfy  $\lambda_{o,y} =\mathcal{O}( \|y\|^{-\alpha})$ for some $\alpha>2d+1$. Then, for each $\epsilon>0$, there exists 
$\theta_0 \in (0,\infty)$ and $N_0 \in \bN$ so that, for all $n\geq N_0$, and for every $A \subset B_n$ and $\delta\geq 0$, 
\begin{equation}\label{eq:AML}
    \bP_{\lambda,\delta} \left( A_t^{A} \subset B_{n\theta_0/2 + t \theta_0}(o) \text{ for all } t>0 \right) \geq 1- \epsilon. 
\end{equation}
\end{proposition}

\begin{proof}
Firstly note that, by monotonicity, it is sufficient to consider the case when $\delta=0$ and $A=B_n$, and where $\lambda \in \Lambda$ is given by  $\lambda_{o,y} =C_{\lambda}\|y\|^{-\alpha}$ for some constant $C_{\lambda}<\infty$.  This corresponds to the first-passage percolation process studied in  \cite{chatterjee2016multiple}. 
The claim therefore follows as an application of \cite[Proposition 8.2]{chatterjee2016multiple}, as we argue next. 

Indeed, let $\kappa=(\alpha-(2d+1))/2>0$. Then, by \cite[Proposition 8.2]{chatterjee2016multiple}, there exists constants $C,c>0$ such that, for all $n \geq 1$,
\begin{align}
\bP_{\lambda,0} \left(  A_t^{B_n} \not  \subset B_{tc}(o) \text{ for all } t \in [n,4n] \right) \leq Cn^{-\kappa}.
\end{align}
Note that, since $\delta=0$, the process never heals and therefore the process $(A_t^{B_n})$ is non-decreasing in $t$. Consequently, we have that
\begin{equation}
\{ B_{cn}(o) \text{ for all } t < n\} \subset \{  A_t^{B_n} \not \subset B_{tc}(o) \text{ for all } t \in [n,4n]\}.
\end{equation} 
Therefore, by also using translation invariance and monotonicity of the process,  it follows that, for any $N \in \bN$ and with  $q_k=(4c)^k$,  
\begin{align}
    & \bP_{\lambda,\delta} \left( A_t^{B_N} \not \subset B_{cN+ 2c t}(o) \text{ for all } t>0 \right)  
    \\ \leq &\sum_{k\geq 0} \bP_{\lambda,0} \left( A_t^{B_{q_kN}} \not \subset B_{q_kN+ t c}(o) \text{ for all } t \in [0,q_{k+1}N]  \right)
    \\ \leq &C N^{-\kappa} \sum_{k\geq 0}  (4c)^{-\kappa k}.
\end{align}
 Observe that the latter term is finite and converges to $0$ as $N\rightarrow \infty$. Thus, there is an $N_0 \in \bN$ so that this probability is below $\epsilon$. From this we conclude the proof by setting $\theta_0=2c$. 
\end{proof}
  
  We are now ready to present the proof of Theorem \ref{thm:fstcWS}. 
  The proof follow the line of reasoning as in Theorem \ref{thm:fstc} closely with some crucial modifications at certain steps in order to deal with that $\lambda$ may not be symmetric.
  
\begin{proof}[Proof of Theorem \ref{thm:fstcWS}] 
Let $\lambda \in \Lambda$  satisfy  $\lambda_{o,y} = \mathcal{O}( \|y\|^{-\alpha})$ for some $\alpha>2d+1$, and let $\delta>0$ so that $\phi(\lambda,\delta)>0$.
Then, for $\epsilon>0$ fixed we let $\tilde{\epsilon}=\epsilon/11$. 
\begin{enumerate}
\item By Proposition \ref{prop:fstcWS1}, since $\alpha>2d+1$, we can fix  $\theta_0 \in [1,\infty)$ and $N_0\in \bN$ so that \eqref{eq:AML} holds with $\epsilon$ replaced by $\tilde{\epsilon}$.
\end{enumerate}

The next three steps of the proof are almost identical to those in the proof of Theorem \ref{thm:fstc}.
\begin{enumerate}
\item[2.] As in Step 1 of the proof of Theorem \ref{thm:fstc} we first fix $A \subset \bZ^d$ so that Proposition \ref{prop1.1} holds with  \eqref{eq:prop1.1_2} having the bound $1-\tilde{\epsilon}^2/2$. Note that  $A$ is now not necessarily a symmetric set. Moreover, we set $\gamma \coloneqq \bP_{\lambda,\delta}( \leftidx{_{A}}|A_1^o \supset A)>0$.

\item[3.] As in Step 2 of the proof of Theorem \ref{thm:fstc}, we let $N' = N'(\gamma)$ be such that $(1-\gamma)^{N'}<\tilde{\epsilon}$. Moreover, we specify $N=N(n,N')$ as therein, with $n=\min(m\geq N_0 \colon A \subset B_m)$, where $N_0 \in \bN$ is as in Step 1.

\item[4.] Similar to  Step 3 of the proof of Theorem \ref{thm:fstc}, we fix $K=K(N,n)$ so large that any $\Delta \subset \bZ^{d+1}$ with $|\Delta|=K$ 
contains a subset of size $2N'$ where all points are at distance at least $2n+2$. 

By possibly increasing $K$, we also assume that $\exp(-2\sqrt{K}) \leq \tilde{\epsilon}/2$ and that $4 \leq K^{1/4}$. Furthermore, we set $K'=16 \lambda_{\infty} K$. 
\end{enumerate}
In the next five steps, we adapt the arguments as in the proof of \cite[Lemma 4.1]{bezuidenhout1994critical} to our setting of long-range interactions.
 
\begin{enumerate}
\item[5.] Fix $t_0 >0$ so that $\bP_{\lambda,\delta}( |A_t^A| \geq N ) \geq 1- \tilde{\epsilon}$ for all $t\geq t_0$. That this is possible follows by Proposition \ref{prop1.1} and corresponds to \cite[Equation (4.4)]{bezuidenhout1994critical}.

\item[6.] We claim that there are $L^{(1)}_d \in [0,\infty)$ and $t_1\geq t_0$ so that, for every subset 
\begin{equation} B_{L',T',\theta} \supset \bZ^{d-1} \times [-L_d^{(1)},L_d^{(1)}] \times [0,t_1],
\end{equation}
where $L'\in \left( [0,\infty) \cup \{\infty\}\right)^d$, $T'\in (0,\infty)$ and $\theta \in \bR^d$,  it holds that
\begin{equation}\label{eq:Eq4.5}
\left(1-\bP_{\lambda,\delta}(| \leftidx{_{B_{L',T'}^{\theta}}}A_T^A | \leq N)\right) \left(1-\bP_{\lambda,\delta}(E_{L',T'}^{\theta,A,+}\geq K')\right) \left(1-\bP_{\lambda,\delta}(E_{L',T'}^{\theta,A,-}\geq K) \right) <\tilde{\epsilon}^3/2.
\end{equation}
Note that here, within the probabilities of \eqref{eq:Eq4.5}, the term $E_{L',T'}^{\theta,A,\pm}$ is the expected number of infection arrows from $I_{L',T'}^{,\theta,A}$ to the right and left in the last coordinate direction, respectively. 
If this is replaced by the cardinality of the infected points in an $2R$-boundary of the box $B_{L',T'}^{\theta}$, denoted by $N_{S_d}^{\pm}$ in \cite{bezuidenhout1994critical}, and where  $R$ is the range of the process considered therein, this  corresponds to \cite[Equation (4.5)]{bezuidenhout1994critical}. Recall that we did a similar adaptation for the proof of Theorem \ref{thm:fstc}, which also there was crucial in order to deal with the long-range setting.  

As in the proof of  Theorem \ref{thm:fstc},  with the above modification, the proof of the inequality \eqref{eq:Eq4.5} goes through as in the proof of \cite[Proposition 3.11]{bezuidenhout1994critical} without any new ideas, by applying arguments similar to those for the proof of Proposition \ref{Prop2}. We therefore omit specifying further details to this argument here. 

Here one should also note that, unlike the statement of  \cite[Proposition 3.11]{bezuidenhout1994critical}, we have set the first $d-1$ coordinates of $L^{(1)}$ equal to $\infty$. However, this does not alter the proof in any way. In fact, it is this version of  \cite[Proposition 3.11]{bezuidenhout1994critical} which is applied in the proof \cite[Lemma 4.1]{bezuidenhout1994critical} and that will be needed in the following arguments.

\item[7.] By Step 5 above, and as in  \cite[Equation (4.6)]{bezuidenhout1994critical}, we may fix $L^{(2)}_d \geq \max (L^{(1)}_d, N_0)$ and $t_2\geq t_1$ so that 
\begin{equation}
\bP_{\lambda,\delta}(| \leftidx{_{L^{(2)}}}A_{t_2}^A | \geq N) \geq 1- \tilde{\epsilon}/2,  
\end{equation}
 and where $L^{(2)}=(\infty,\dots,\infty,L^{(2)}_d)$.

\item[8.] Fix a non-integer number  $L^{(3)}=(\infty,\dots,\infty,L^{(3)}_d)$ with $L^{(3)}_d \geq \theta_0 L^{(2)}_d/2 + t_2 \theta_0$.  

This choice has two important consequences. Firstly, for any box $B_{L^{(3)},t_3}^{\theta}$ with $t_3\geq t_2$ and $|\theta|>\theta_0$ that does not fully contain $B_{L^{(2)}, t_2}$, it holds by \eqref{eq:AML} that
\begin{equation}
\min \left(\bP_{\lambda,\delta}(E_{L^{(3)},t_3}^{\theta,A,+}\geq K'), \bP_{\lambda,\delta}(E_{L^{(3)},t_3}^{\theta,A,-}\geq K') ) \right)<1-\tilde{\epsilon}.
\end{equation}
Secondly, whenever $|\theta| \leq \theta_0$, we have that $B_{L^{(2)},t_2} \subset B_{L^{(3)},t_3}^{\theta}$ and so the claim of Step 6 apply to any such space-time box.

\item[9.] The choice of $L^{(3)}$ in the previous step corresponds to that of $w_2$ in \cite[Equation (4.9)]{bezuidenhout1994critical}. 
The remaining arguments for the proof of \cite[Lemma 4.1]{bezuidenhout1994critical}, particularly those for the proof of Claim 4.10 and Claim 4.16 therein, albeit somewhat technical, applies  essentially word by word to our long-range setting after minor notational change (i.e.\ replacing the objects $N_{S_d}^{\pm}$ by $E_{L^{(3)},t_3}^{\theta,A,\pm}$ and noting that $ | \leftidx{_L^{(2)}}A_{t_2}^A | \geq N$ corresponds to $N_T(B_{L^{(2)}},t_2)\geq N$). Therefore, referring to \cite{bezuidenhout1994critical} for further details to this argument, we conclude the following: there exists a box $B_{L,T}^{\theta}$ with $L\in (0,\infty)^d$, $T<\infty$ and $ \theta=(0,\dots,0,\theta_d)$ satisfying $|\theta_d|\leq \theta_0$ such that
\begin{align}\label{eq:lem4.1a} 
& \bP_{\lambda,\delta} \left(| \leftidx{_{B_{L,T}^{\theta}}}A_{T}^{A} | >  N \right) > 1-\tilde{\epsilon};
\\ &     \bP_{\lambda,\delta} \left( E_{L,T}^{\theta,A,\pm}>K' \right) > 1-2\tilde{\epsilon}, \label{eq:lem4.1b}
\end{align}
and where the coordinates $L$ and the quantity $T$ can be chosen as large as we wish.
\end{enumerate}

With the identities obtained in Step 9, and letting $L$, $T$ and $\theta$ be such that \eqref{eq:lem4.1a}-\eqref{eq:lem4.1b}  hold with $L\geq \max(2n,n\theta_0/2)$ and $T\geq 1$, we now explain how to complete the proof by applying arguments as in Steps 5-7 of the proof of Theorem \ref{thm:fstc}.

\begin{enumerate}
\item[10.] By essentially the same argument as in  Step 5 of the proof of Theorem \ref{thm:fstc} we conclude from \eqref{eq:lem4.1a} and since $L\geq 2n$ that 
\begin{equation}\label{eq:block1NS}
 \bP_{\lambda,\delta} \left(\exists x\in B_{L,T}^{\theta}  \colon A(x) \subset \leftidx{_{B_{2L,T+1}^{\theta}}}{A}{_{T+1}^{A}} \right) \geq (1-\tilde{\epsilon})^2,
 \end{equation}
 which by our choice of $\tilde{\epsilon}$ implies \eqref{C3}.
 \item[11.] Also the arguments as in Step 6 of the proof of Theorem \ref{thm:fstc} can be modified to the non-symmetric setting. 
 Indeed, similar to therein, for $R>1$ let $M_{L,T,R}^{\theta,A,\pm}$ denote the maximal number of points within $S_{L,T,R}^{\theta,A,\pm}$ to which there is an infection arrow from a space-time point in $I_{L,T}^{\theta,A}$ and  complying with the following: any two distinct points  $(x,s)$ and $(y,t)$ contained in $M_{L,T,R}^{\theta,A,\pm}$  satisfies either that $|s-t|> 1$ or that $\|x-y\|_{\infty} \geq 2n+1$. 
 
 Conditioned on the event in \eqref{eq:lem4.1b} to hold, we claim that the event 
   \begin{equation}\label{eq:Step11_M}
  \left\{ M_{L,T,2\theta_0 (L+T)}^{\theta,A,\pm}> N' \right\}
  \end{equation}
  occurs with a probability of at least $1-8\tilde{\epsilon}$. 
  To see this, note first that, by  Proposition \ref{prop:fstcWS1} and the inequality \eqref{eq:lem4.1b}, and using that $L\geq n\theta_0/2$, we have 
    \begin{equation}\label{eq:claimBlokkis_new}
       \bP_{\lambda,\delta} \left( E_{L,T, 2\theta_0 (L+T)}^{\theta,A,\pm}>K' \right) > 1-4\tilde{\epsilon}.
  \end{equation}
  This is the key inequality in Step 6 of the proof of Theorem \ref{thm:fstc} that needs to be modified. Indeed, replacing  \eqref{eq:claimBlokkis} in Step 6 of the proof of Theorem \ref{thm:fstc} with this new estimate of \eqref{eq:claimBlokkis_new}, with the our choice of $K$ and $K'$ as in Step 4 above, we can now proceed exactly as in that proof to obtain the above claimed inequality for the event in \eqref{eq:Step11_M}. 

  \item[12.] As a consequence of the previous step, by an argument reminiscent of that in Step 7 of the proof of Theorem \ref{thm:fstc}, we have that
  \begin{equation}\label{eq:block2NS}
   \bP_{\lambda,\delta} \left( \exists (x,t)  \in S_{L,T,2\theta_0( L+T)}^{\theta,A,\pm} \colon A(x) \subset  \leftidx{_{B_{2(1+2\theta_0)(L+T),T+1}^{\theta}}}{A}{_{t+1}^{A}} \right) \geq (1-8\tilde{\epsilon}) 
   (1-2\tilde{\epsilon})(1-\tilde{\epsilon}),
  \end{equation}
  i.e., by our choice of $\tilde{\epsilon}$, the inequality \eqref{C4} holds with $r=2\theta_0$. Indeed, the factor $1-2\tilde{\epsilon}$ corresponds to \eqref{eq:lem4.1b}. The factor $1-8\tilde{\epsilon}$  corresponds to the probability that then also \eqref{eq:Step11_M} holds. In that event, by our choice of $N'$ and independence and translation invariance of  the graphical representation within disjoint space-time regions, with a a probability of at least $1-\tilde{\epsilon}$ the event in  \eqref{eq:block2NS} occurs.
\end{enumerate}
\end{proof}

\section{Resilient parameters}\label{sec:proofs}
In this section we present the proofs of Theorem \ref{thm:resilient2} and Theorem \ref{thm:resilient}.  In the last subsection we also discuss briefly some further applications of the resilient property, i.e.\ local survival, complete convergence and the shape theorem.

\subsection{Proof of Theorem \ref{thm:resilient2}}

Theorem \ref{thm:resilient2} follows by arguments similar to those in \cite{bezuidenhout1990critical}, also utilising the extension to the finite-range case in \cite{bezuidenhout1994critical}.

\begin{proof}[Proof of Theorem \ref{thm:resilient2}]
Firstly, assume that $\chi(\lambda,\delta)=\infty$ and $\phi(\lambda,\delta)=0$. Then, by Proposition \ref{prop:SubCritExpDecay}i), the corresponding LRCP is not subcritical, and by Proposition \ref{prop:SubCritExpDecay}iii) it is not supercritical. Hence, it is critical.

Now, let $\lambda \in \Lambda$ be resilient and assume that $\delta = \delta_c(\lambda)$, i.e.\ that the corresponding LRCP is critical. Then, by  Proposition \ref{prop:SubCritExpDecay}ii), we know that $\chi(\lambda,\delta)=\infty$ and so what remains to be shown is that $\phi(\lambda,\delta)=0$. That this holds follows by a minor adaptation of the proof given in the case of the ordinary contact process, first proved in \cite{bezuidenhout1990critical}.

Indeed,  suppose that $\phi(\lambda,\delta)>0$. Then, since $\lambda$ is resilient, we have that $\phi(\leftidx{_{k}}\lambda,\delta)>0$ for all $k\in \bN$ large. Consequently, since $\leftidx{_{k}}\lambda$ is finite range, by \cite{bezuidenhout1994critical} the finite space-time condition as in Theorem \ref{thm:fstcWS} holds. Moreover, since for any fixed $\epsilon>0$ the probabilities in \eqref{C3} and \eqref{C4} are continuous with respect to the parameters of the LRCP, there is a $c>1$ such that the finite space-time condition also holds for the LRCP with parameters $\leftidx{_{k}}\lambda$ and $c \delta$. To see this, consider the natural monotone coupling of the process with infection parameter $\leftidx{_{k}}\lambda$ and healing parameter $\delta$ and $c\delta$, respectively. Then, with a probability converging to $1$ as $c\rightarrow 1$ the two process share exactly the same healing events in the space-time boxes as in \eqref{C3} and \eqref{C4}, and in that event the two processes are identical. Therefore, since Theorem \ref{thm:fstcWS} holds, by the work of \cite{bezuidenhout1994critical} it follows that $\phi(\leftidx{_{k}}\lambda,c\delta)>0$. Hence, by monotonicity of the LRCP, also $\phi(\lambda,c\delta)>0$. But this contradicts the assumption that the process is critical. Thus, we conclude that $\phi(\lambda,\delta)=0$. 
\end{proof}

\subsection{Proof of Theorem \ref{thm:resilient}}

As a first step, we argue that  $\phi(\cdot, \cdot)$ is right-continuous.

\begin{proposition}\label{prop:rc}
Let $\lambda \in \Lambda$ and $\delta>0$.   Then, for any two sequences $(\lambda^{(n)})$ in $\Lambda$ and $\delta^{(n)}$ in $[0,\infty)$ satisfying  that
$\lambda^{(n)} \downarrow \lambda$ and $\delta^{(n)} \uparrow \delta$, 
it holds that $\lim_{n \rightarrow \infty} \phi(\lambda^{(n)},\delta^{(n)}) = \phi(\lambda,\delta).$
\end{proposition}
\begin{proof}
This follows by the monotonicity of the model and the graphical construction as in Section \ref{sec:preliminiaries}, using that 
$\phi(\lambda,\delta) = \bP_{\lambda,\delta} (\forall t<0 \exists y \in \bZ^d \colon (y,t) \rightarrow (o,0))$ 
for any $\lambda \in \Lambda$ and $\delta >0$. 
 Indeed, for each $n \in \bN$, we have that 
\begin{align} 
\phi(\lambda^{(n)},\delta^{(n)}) 
 &=  \bP_{\lambda^{(n)},\delta^{(n)}} \left(\forall t<0 \exists y \in \bZ^d \colon (y,t) \rightarrow (o,0)\right)  
\\ &= \lim_{T \rightarrow \infty} \bP_{\lambda^{(n)},\delta^{(n)}} \left(\exists y \in \bZ^d \colon (y,-T) \rightarrow (o,0) \right).
\end{align}
Now, note that $\phi(\lambda^{(n)},\delta^{(n)})$ is non-increasing in $n$ since,  for any $T>0$ and $n\geq 1$, 
\begin{equation}
\bP_{\lambda^{(n)},\delta^{(n)}} \left(\exists y \in \bZ^d \colon (y,-T) \rightarrow (o,0) \right) \geq \bP_{\lambda^{(n+1)},\delta^{(n+1)}} \left(\exists y \in \bZ^d \colon (y,-T) \rightarrow (o,0) \right),
\end{equation}
as follows directly by the natural monotone coupling. Thus, also $\phi(\lambda^{(n)},\delta^{(n)})\geq \phi(\lambda,\delta)$. 

On the other hand, we claim that, for any $\epsilon>0$ and any $T>0$ large, 
\begin{equation}\label{eq:rightCont1}
\lim_{n \rightarrow \infty} \bP_{\lambda^{(n)},\delta^{(n)}} \left(\exists y \in \bZ^d \colon (y,-T) \rightarrow (o,0)\right) \leq \phi(\lambda,\delta)+ \epsilon.
\end{equation}
To see this, fix first $T>0$ so large that
\begin{equation}\label{eq:rightContHelp2}
\phi(\lambda,\delta) \geq \bP_{\lambda,\delta} \left(\exists y \in \bZ^d \colon (y,-T) \rightarrow (o,0)\right) -\epsilon/2
\end{equation}
Then, since $\lambda^{(1)} \in \Lambda$ and by monotonicity, it follows by \cite[Theorem 1.2ii)]{chatterjee2016multiple} and the Markov inequality that, for any $t>T$ there is  an $l=l(t)$ such that
\begin{equation}
\widehat{\bP} \left( A_s^{(n)} \subset B_l \text{ for all } s \in [0,t] \text{ for all } n \right) \geq 1-\epsilon/4, 
\end{equation}
where $A_s^{(n)}$ denotes the process with parameters $\lambda^{(n)}$ and $\delta^{(n)}$, and $\widehat{\bP}$ is the natural monotone coupling. 
Note also that, by self-duality of the LRCP, we have that
\begin{equation}
\bP_{\lambda^{(n)},\delta^{(n)}} \left(\exists y \in \bZ^d \colon (y,-t) \rightarrow (o,0)\right) = \bP_{\lambda^{(n)},\delta^{(n)}} \left(\exists y \in \bZ^d \colon (o,0) \rightarrow (y,t)\right).
\end{equation}
Hence, by continuity of the processes restricted to infection and healing events of the graphical construction contained in $B_{l,t}$, we have that, for all $n$ large,
\begin{equation}\label{eq:rightContHelp1}
\bP_{\lambda^{(n)},\delta^{(n)}} \left(\exists y \in \bZ^d \colon (o,0) \rightarrow (y,t)\right) \leq \bP_{\lambda,\delta} \left(\exists y \in \bZ^d \colon (o,0) \rightarrow_{B_{l,t}} (y,t)\right) + \epsilon/2.
\end{equation}
Combining \eqref{eq:rightContHelp1} with \eqref{eq:rightContHelp2} we thus obtain \eqref{eq:rightCont1}, and from which, since $\epsilon>0$ was arbitrary and $\phi(\lambda^{(n)},\delta^{(n)})\geq \phi(\lambda,\delta)$, we conclude the proof. 
\end{proof}

Before providing the proof of Theorem \ref{thm:resilient}, we present the following key estimate in order to show that $\phi(\cdot,\cdot)$ also is left-continuous. 

\begin{proposition}\label{prop:tue}
   Let $\lambda \in \Lambda$ be resilient. Then for any $\delta<\delta_c(\lambda)$ there exists constants $C,c \in (0,\infty)$ such that, for any $t>0$ and $\hat{\lambda} \geq \lambda$, 
    \begin{equation}\label{eq:tue}
        \bP_{\hat{\lambda},\delta} \left( t \leq \tau^o <\infty \right) \leq Ce^{-ct}
    \end{equation}
    where $\tau^o \coloneqq \inf \{ t >0 \colon \eta_t^o=\underline{0} \}$ is the time until extinction of the infection process. 
\end{proposition}

\begin{proof}
Firstly, since $\delta<\delta_c(\lambda)$, we have that $\phi(\lambda,\delta)>0$. 
If $\lambda$ is a nearest neighbour infection parameter, the statement of Proposition \ref{prop:tue} therefore corresponds to \cite[Theorem 2.30a)]{LiggettSIS1999}. The proof in that case follows by a restart argument and a coupling comparison with a supercritical  oriented percolation process. This argument also extends to  the case where $\lambda$ is of finite range. Indeed, as concluded in \cite[Proposition 5.1]{bezuidenhout1994critical} (for a more general class of models), for $\epsilon>0$ (after possibly making a linear change of space-time coordinates that leaves the time-coordinate fixed) one can find a finite set $A \subset \bZ^d$ and $L \in[0,\infty)^d$ and $T\in (0,\infty)$ such that  
        \begin{align}
& \bP_{\lambda,\delta} \left(\exists (y,t) \in R_{\pm} \colon  A(y)  \subset \leftidx{_{L}}{A}{_{T}^{D}} \right)> (1-\epsilon);
 \end{align}
 where $R_{\pm}= \{(x,t)\in B_{L,T} \colon (\pm x_d,t) \in [L_d/3,L_d] \times [T/3,T]\}$. This is in turn applied to provide a coupling comparison with a supercritical  oriented percolation process, in the same vein as for the nearest neighbour model. This is detailed in the proof of \cite[Theorem 2.4]{bezuidenhout1994critical} in Section 6. With this coupling comparison, the proof of \eqref{eq:tue} detailed in \cite[Theorem 2.30a)]{LiggettSIS1999} extends to the finite range case after only minor notational changes.
    
   Now, assume that $\lambda$ is of infinite range. Then, since it is assumed to be resilient, as in the proof of Theorem \ref{thm:resilient2}, we have that $\phi(\leftidx{_{k}}\lambda,\delta)>0$ for some $k \in \bN$, where $\leftidx{_{k}}\lambda$ is defined in \eqref{eq:lambdaKcK}. Particularly, $\leftidx{_{k}}\lambda$ is of finite range and therefore, as argued above, it satisfies \eqref{eq:tue}. To extend this inequality to $\lambda$ again goes by another restart argument that we detail next. 
   
   Consider the natural monotone coupling of the LRCP with infection parameters $\lambda$ and $\leftidx{_{k}}\lambda$ and common parameter $\delta$.
   Let $v_0=0$, and define iteratively $(v_l)_{l\geq1}$ as follows. If $\eta_{v_l}^o=\emptyset$, then we set $v_{l+1}=v_l$. Otherwise, if $\eta_{v_l}^o\neq\emptyset$, let $x_l$ be a vertex with the property that $\eta_{v_l}^o(x_l)=1$, chosen in an arbitrary manner. Let $(\eta_{v_l +t}^{(l)})$ denote the truncated LRCP started at time $v_l$ with only individual $x_l$ infected and only using the arrows stemming from the  $\leftidx{_{k}}\lambda$ parameter. Consequently,
we have that $\eta_{v_l +t}^{(l)}\leq \eta_{v_l +t}^{o}$ for all $t\geq0$ a.s. Now, set $v_{l+1} = \inf \{ s\geq v_k \colon \eta_{v_l +s}^{(l)} = \emptyset \}$ and let 
\begin{equation}
l_0 := \inf \left\{ l\geq 0 \colon v_{l}=v_{l+1} \text{ or } v_{l+1}=\infty \right\}.
\end{equation}
Thus, since $\phi(\leftidx{_{k}}\lambda,\delta)>0$, on each trial time $l$ there is a positive probability that $v_l=\infty$. In particular, $l_0$ is majorised by a geometrically distributed random variable and, moreover,  we have that 
\begin{equation}
\left\{ t < \tau^o < \infty\} = \{t < v_{l_0} < \infty\right\}.
\end{equation}
Further, from the graphical construction we have that, given $l_0=l$, the times $v_j-v_{j-1}$ for $1\leq j \leq l$ are i.i.d.  Moreover, they have exponential tails since \eqref{eq:tue} holds  for the LRCP with parameter $\leftidx{_{k}}\lambda$. 
By combining these two properties, we conclude the proof of the proposition for the LRCP with parameters $\lambda$ and $\delta$. Moreover, by monotonicity, we note that the last argument still applies if we instead consider the process with infection parameter $\hat{\lambda}\geq \lambda$, yielding exactly the same bound. 
\end{proof}

Finally, utilising the two propositions above, we now present the proof of Theorem \ref{thm:resilient}.

\begin{proof}[Proof of Theorem \ref{thm:resilient}]
    Let $\lambda \in \Lambda$ be resilient and $\delta>0$, and consider two sequences $\lambda^{(n)} \in  \Lambda$ and  $\delta^{(n)}$ in $[0,\infty)$ satisfying \eqref{eq:condTheoremResilient}. 
        Further, for $n\in \bN$, let $\lambda^{(n,+)}, \lambda^{(n,-)}  \in \Lambda$ be given by
    \begin{equation}
    \lambda^{(n,+)}_{o,y} \coloneqq \sup( \lambda_{o,y}^{(m)}, m\geq n ), \quad \lambda^{(n,+)}_{o,y} \coloneqq \inf( \lambda_{o,y}^{(m)}, m\geq n ),
    \end{equation} 
    and let $\delta^{n,+} = \sup( \delta^{(m)}, m\geq n)$ and $\delta^{n,-} = \inf( \delta^{(m)}, m\geq n)$.
    
    As a direct consequence of Proposition  \ref{prop:rc}, using that $\lambda^{(n,+)} \downarrow \lambda$ and $\delta^{(n,-)} \uparrow \delta$, we have that 
    \begin{equation}\label{eq:thm:rc}
    \lim_{n\rightarrow \infty} \phi( \lambda^{(n,+)}_{o,y},\delta^{(n,-)}) =\phi(\lambda,\delta).
    \end{equation}

    For $\delta\geq\delta_c(\lambda)$, we have that $\phi(\lambda,\delta)=0$, as follows by Proposition \ref{prop:SubCritExpDecay}i) and Theorem \ref{thm:resilient2}.
   Therefore, since $\lambda^{(n)} \leq \lambda^{(n,+)}$ and $\delta^{(n)}\geq \delta^{(n,-)}$, by monotonicity and since  \eqref{eq:thm:rc} holds, we conclude in this case that 
   \begin{equation}\lim_{n\rightarrow \infty} \phi( \lambda^{(n)}_{o,y},\delta^{(n)}) =0.
   \end{equation}
        
    Assume now  that $\delta<\delta_c(\lambda)$ so that $\phi(\lambda,\delta)>0$ by Proposition \ref{prop:SubCritExpDecay}iii). 
 Then, since $\lambda$ is resilient, there is $K \in \bN$ such that also $\phi(\leftidx{_{k}}\lambda,\delta)>0$ for all $k \geq K$. 
 Moreover, for any such $k\geq K$ fixed, we also have that $ \phi( \leftidx{_{k}}\lambda^{(n,-)}, \delta^{(n,+)})>0$  for all $n=n(k)$ large. 
  This follows by continuity of the infection and healing events, by a similar argument as in the proof of Theorem \ref{thm:resilient2}. 
Particularly, by monotonicity, it holds that $ \phi( \lambda^{(n,-)}, \delta^{(n,+)})>0$ for all $n$ large.

 Denote by $\widehat{\bP}_n$ the natural monotone coupling of the processes $(A_t^{(1)})$ and $(A_t^{(2)})$ with parameters $(\lambda^{(n,-)},\delta^{(n,+)})$ and $(\lambda,\delta)$, respectively, obtained via the graphical construction and both having initially only the origin in the infected state.   Then, for any $T>0$, 
\begin{align}
   | \phi(\lambda^{(n,-)},\delta^{(n,+)})-\phi(\lambda,\delta)| 
   &\leq \widehat{\bP}\left( A_T^{(1)} =\emptyset, A_T^{(2)}  \neq \emptyset \right) 
   \\&+ \widehat{\bP}\left( A_T^{(1)} \neq \emptyset, A_t^{(1)} =\emptyset \text{ for some } t>T  \right).
\end{align}
Now, for $\epsilon>0$ and any fixed $T$, by continuity of LRCP within a finite time-span, it follows that
\begin{equation}\label{eq:thm1.3_helpme}
 \widehat{\bP}\left( A_T^{(1)} =\emptyset, A_T^{(2)}  \neq \emptyset \right) \leq \epsilon/2
\end{equation}
for all $n=n(T)$ large. 
Moreover, by Proposition \ref{prop:tue}, 
there are uniform constants $C,c \in (0,\infty)$ that do not depend on $T$ nor on $n$ such that
\begin{equation}
\widehat{\bP}\left( A_T^{(1)} \neq \emptyset, A_t^{(1)} =\emptyset \text{ for some } t>T  \right) \leq Ce^{-cT}
\end{equation}  
Therefore, by first tuning $T$ large such that $Ce^{-cT}\leq \epsilon/2$ and then tuning $n$ large such that \eqref{eq:thm1.3_helpme} holds, we have that 
\begin{equation}| \phi(\lambda^{(n,-)},\delta^{(n,+)})-\phi(\lambda,\delta)| \leq \epsilon.
\end{equation} 
Hence, since $\epsilon>0$ was arbitrary, $\lambda^{(n,-)}\uparrow \lambda$ and $\delta^{(n,+)}\downarrow \delta$,  we conclude that 
\begin{equation}\phi(\lambda^{(n,-)},\delta^{(n,+)})\uparrow \phi(\lambda,\delta) \quad \text{ as }n\rightarrow \infty. \end{equation}  
From this and \eqref{eq:thm:rc}, and since 
$\lambda^{(n,-)}\leq \lambda^{(n)} \leq \lambda^{(n,+)}$ and $\delta^{(n,-)} \leq \delta^{(n)} \leq \delta^{(n,+)}$,  
we conclude that $\lim_{n \rightarrow \infty } \phi( \lambda^{(n)}, \delta^{(n)})= \phi(\lambda,\delta)$.

Lastly, to see that $\delta_c(\lambda^{(n)}) \rightarrow \delta_c(\lambda)$ as $n \rightarrow \infty$, note that, for  any $\epsilon>0$, 
\begin{equation}
\lim_{n \rightarrow \infty} \phi(\lambda^{(n)}, (1-\epsilon) \delta_c(\lambda)) =  \phi(\lambda, (1-\epsilon) \delta_c(\lambda)) >0.
\end{equation}
Thus, $\phi(\lambda^{(n)}, (1-\epsilon) \delta_c(\lambda)) >0 $ for all $n$ large, and so $\delta_c(\lambda^{(n)}) \geq (1-\epsilon) \delta_c(\lambda)$. 
On the other hand, since the LRCP with parameters $\lambda$ and $(1+\epsilon)\delta_c(\lambda)$ is subcritical, we know from Proposition \ref{prop:SubCritExpDecay}i) that, for some $\tau>0$ and $T>0$,
\begin{equation}\label{eq:lastexpectation}
 \frac{1}{T} \log  \bE_{\lambda,(1+\epsilon)\delta_c(\lambda)} \left[  \left| \left\{ x\in \bZ^d \colon \eta_T^o(x)=1\right\} \right| \right] \leq - \tau/2.
 \end{equation}
But, since $T$ is fixed, the expectation in \eqref{eq:lastexpectation} is continuous with respect to the parameter $\lambda$. Therefore also
\begin{equation}
\lim_{n \rightarrow \infty} \frac{1}{T} \log  \bE_{\lambda^{(n)},(1+\epsilon)\delta_c(\lambda)} \left[  \left| \left\{ x\in \bZ^d \colon \eta_T^o(x)=1\right\} \right| \right] \leq - \tau/2
\end{equation}
From this and \cite[Theorem 1]{Swart2018} we conclude that $\delta_c(\lambda^{(n)}) \leq (1+\epsilon) \delta_c(\lambda)$ for all $n$ large. 
This implies the claim since also $\delta_c(\lambda^{(n)}) \geq (1-\epsilon) \delta_c(\lambda)$ for all $n$ large, as concluded above.
\end{proof}

\subsection{Local survival, complete convergence and the shape theorem}

In this last subsection, we briefly discuss some further applications for resilient LRCP.

\textbf{Local survival.} 
As common for the ordinary contact process, we say that the LRCP  \emph{survive globally} if  $\phi(\lambda,\delta)>0$. Moreover, it \emph{survive locally} if $\beta(\lambda,\delta)>0$ where $\beta \colon \Lambda \times (0,\infty) \to [0,1]$ is given by 
\begin{equation}
\beta(\lambda,\delta) \coloneqq \inf_{T>0} \bP_{\lambda,\delta}( \eta_t^o(o) = 1 \text{ for  some } t>T).
\end{equation}
For the ordinary contact process on $\bZ^d$,  \cite[Theorem I.2.25a)]{LiggettSIS1999} states that $\phi(\lambda,\delta)>0$ if and only if $\beta(\lambda,\delta)>0$, where the only if is the nontrivial statement. This extends to the LRCP with parameters $\lambda \in \Lambda$ and $\delta>0$ under the additional assumptions that $\lambda$ is resilient, symmetric and irreducible. Indeed, consider $\lambda\in \Lambda$ and $\delta>0$ so that $\phi(\lambda,\delta)>0$. Then, since $\lambda$ is resilient, also $\phi(\leftidx{_{k}}\lambda,\delta)>0$ for some $k \in \bN$ under which $\leftidx{_{k}}\lambda$ is symmetric and irreducible. Then, by applying Proposition \ref{Prop3}, one can easily adapt our argument to conclude that the statement of Theorem \ref{thm:fstc} holds with $A=[-n,n]^d$ for some $n \in \bN$, as in \cite[Theorem I.2.12)]{LiggettSIS1999}. With this minor modification, the proof of \cite[Theorem I.2.25a)]{LiggettSIS1999} extends without further changes to conclude that $\beta(\leftidx{_{k}}\lambda)>0$ and therefore also $\beta(\lambda)>0$ by monotonicity. 

\textbf{Complete convergence.}  
For parameters $\lambda \in \Lambda$ and $\delta>0$, consider the function $\rho \colon \{0,1\}^{\bZ^d}\to [0,1]$  given by $\rho(\omega) \coloneqq \bP_{\lambda,\delta}(\eta_t^{\omega} \neq \underline{0}  \text{ for some } t>0).$ 
Then, whenever $\lambda$ is resilient, symmetric and irreducible, 
\begin{equation}\label{eq:cc}
    \bP_{\lambda,\delta}(\eta_t^{\omega}\in \cdot)\implies 
    \rho(\omega) \delta_{\underline{0}}(\cdot) + (1-\rho(\omega)) \bar{\nu}_{\lambda,\delta}(\cdot) \quad \text{ as }t\rightarrow \infty,
\end{equation}
for any $\omega \in \{0,1\}^{\bZ^d}$, where $\implies$ denotes weak converges. This follows by the methods of \cite{bezuidenhout1990critical} for the  ordinary contact process, see e.g.\ \cite[Theorem I.2.27)]{LiggettSIS1999} for a proof. Note that the only non-trivial case is when $\phi(\lambda,\delta)>0$, i.e.\ within the supercritical regime.  Then, again since $\lambda$ is resilient, also $\phi(\leftidx{_{k}}\lambda,\delta)>0$ for some $k \in \bN$ under which $\leftidx{_{k}}\lambda$ is symmetric and irreducible. Hence, the statement of Theorem \ref{thm:fstc} holds with $A=[-n,n]^d$ for some $n \in \bN$. Using this, and further utilising the basic properties obtained from the graphical construction, the proof e.g.\ as in \cite{LiggettSIS1999} goes through essentially verbatim.

\textbf{The shape theorem.} 
 Consider the LRCP with  $\lambda \in \Lambda$ and $\delta>0$ such that $\phi(\lambda,\delta)>0$. Denote by $B_t$ the set of vertices that were infected within the time-window $[0,t]$, i.e.\
\begin{equation}
    B_t = \{x \in \bZ^d \colon \eta_s^o(x)=1 \text{ for some } s<t \}.
\end{equation}
Further, let $\tilde{B}_t = B_t + [0,1]^d \subset \bR^d$ be the set of points contained within distance $1$ of a lattice point in $B_t$. Then the so-called shape theorem states that there exists a non-trivial convex set $U \in \bR^d$ such that
\begin{equation}
\bar{\bP}_{\lambda,\delta} \left( (1-\epsilon)U \subset \frac{1}{t} \tilde{B}_t \subset (1+\epsilon){U} \: \forall \: t \text{ large} \right) = 1,
\end{equation}
where $\bar{\bP}_{\lambda}= \bP_{\lambda}( \cdot | \tau^o=\infty )$ with $\tau^o = \inf \{ t > \colon A_t^o = \emptyset \}$. We also write $\tau(y) = \inf\{ t >0 \colon y \in A_t^o \}$, $y \in \bZ^d$. 

For the nearest neighbour model on $\bZ^d$, $d\geq 1$, the shape theorem follows by the work of \cite{bezuidenhout1990critical}, see also \cite{DurrettGriffeath1982}. Further, assuming irreducibility and symmetry, in the finite range case it follows due to the work of \cite{bezuidenhout1994critical} e.g.\ as then all of the general conditions of  random linear growth models in  \cite[Theorem 2]{deshayes2025contact} are satisfied.  In particular, there are constants $C_1,C_2,M_1,M_2>0$ such that for all $t\in [0,\infty)$ and $x \in \bZ^d$, 
\begin{align}
&\text{(AML)} &\bP_{\lambda,\delta} \left(\exists y \in \bZ^d \colon \tau(y)\leq t \text{ and } \|y\| \geq M_1 \right) \leq C_1\exp(-C_2 t),
\\ &\text{(SC)} &\bP_{\lambda,\delta} \left(t <t^o < \infty \right) \leq C_1\exp(-C_2 t),
\\ &\text{(ALL)} &\bP_{\lambda,\delta} \left( \tau(y)\geq M_1\|x\| + t, \tau^o=\infty \right) \leq C_1\exp(-C_2 t).
\end{align}
In fact, (SC) holds due to Proposition  \ref{prop:tue}, whereas \cite[Proposition 1.1ii)]{AizenmanJung2007} implies that the LRCP satisfies the (AML), even when $\lambda \in \Lambda$ is of infinite range decaying at an exponential rate. Further, by the standard coupling with an oriented percolation model as in the proof of Theorem \ref{thm:truncation}, utilising that Theorem \ref{thm:fstc} holds with $A=[-n,n]^d$ for some $n \in \bN$, it can be shown that the LRCP  satisfies the (ALL) too for $\lambda$ irreducibility and symmetry.

We believe that the shape theorem can be further extended to the LRCP with infinite range infection parameter. Particularly, all the above estimates hold  when $
\lambda$ decays at an exponential rate, and we expect that the technique of \cite[Theorem 2]{deshayes2025contact}  extends to this setting without much further ado. 

The shape theorem was shown to hold for the corresponding first-passage percolation process, i.e.\ the LRCP with $\delta=0$, in \cite[Theorem 1.7]{chatterjee2016multiple} for $\lambda$ of the form $\lambda_{o,y} =\mathcal{O}(\|y\|^{-\alpha})$ for some $\alpha>2d+1$. We anticipate that it can also be extended to the LRCP,  and possibly it may even be sharpened to hold for $\lambda$ with decay rate $\alpha=2d+1$. In that case it is only the (AML) bound of those above that fails. However, we do not expect the shape theorem to hold when $\alpha< 2d+1$ since it then fails for the first-passage percolation process, as also concluded in \cite{chatterjee2016multiple}.


\paragraph{Acknowledgement and Funding information.} This work has been supported by the Mathematics for Sustainable Development (MATH4SDG) project, which is a research and development project running in the period 2021–2026 at Makerere University-Uganda, the University of Dar es Salaam-Tanzania, and the University of Bergen-Norway, funded through the NORHED II program under the Norwegian Agency for Development Cooperation (NORAD, project no 68105).

Frank Namugera also acknowledge financial support from Erasmus +, and thanks the University of Groningen for hospitality during the writing of the paper. This work
is part of his PhD thesis, which he obtained at Makerere University in December 2025. 

\end{document}